\documentclass[a4paper,oneside]{article}
\usepackage{amsmath, amssymb, amsthm,graphicx}
\usepackage{xcolor,color}
\usepackage[font=small,labelfont=md,textfont=it]{caption}
\usepackage{floatrow,float}
\usepackage[titletoc, title]{appendix}
\usepackage[colorlinks,linkcolor=blue,citecolor=blue]{hyperref}
\usepackage{longtable}
\usepackage{pgfplots}
\usepackage{diagbox}
\usepackage{booktabs,makecell,multirow}
\usepackage[capitalise, nosort]{cleveref}
\usepackage{algorithm}
\usepackage{algorithmic}
\usepackage{amsmath,bm}
\usepackage{cases}
\usepackage{geometry}
\usepackage{extarrows}
\usepackage{tcolorbox}
\usepackage{bbm}
\usepackage{syntonly}
\usepackage[figuresright]{rotating}

\crefname{equation}{}{}
\crefname{lem}{Lemma}{Lemmas}
\crefname{thm}{Theorem}{Theorems}

\newcommand{\id}{\,{\rm Id}}

\newcommand{\dd}{\,{\rm d}}
\newcommand{\R}{\,{\mathbb R}}

\newcommand{\dual}[1]{\left\langle {#1} \right\rangle}
\newcommand{\prox}[0]{ {\bf prox}}
\newcommand{\proxi}[0]{ {\bf prox}}

\newcommand{\argmin}[0]{ {\mathop{{\rm  argmin}}\,}}

\newcommand{\argmax}[0]{ {\mathop{{\rm  argmax}}\,}}

\newcommand{\nm}[1]{\left\lVert {#1} \right\rVert}
\newcommand{\snm}[1]{\left\lvert {#1} \right\rvert}
\newcommand{\ssnm}[1]
{
	\left\vert\kern-0.25ex
	\left\vert\kern-0.25ex
	\left\vert
	{#1}
	\right\vert\kern-0.25ex
	\right\vert\kern-0.25ex
	\right\vert
}

\makeatletter
\def\spher@harm#1{%
	\vbox{\hbox{%
			\offinterlineskip
			\valign{&\hb@xt@2\p@{\hss$##$\hss}\vskip.2ex\cr#1\crcr}%
		}\vskip-.36ex}%
}
\def\gshone{\spher@harm{.}}
\def\gshtwo{\spher@harm{.&.}}
\def\gshthree{\spher@harm{.&.&.}}
\let\gsh\spher@harm
\makeatother

\newtheorem{coro}{Corollary}[section]

\newtheorem{lem}{Lemma}[section]

\newtheorem{rem}{Remark}[section]

\newtheorem{thm}{Theorem}[section]

\newcounter{mnote}
\let\oldmarginpar\marginpar
\renewcommand\marginpar[1]
{\-\oldmarginpar[\raggedleft\footnotesize #1]{\raggedright\footnotesize #1}}

\makeatletter\def\@captype{table}\makeatother

\newfloatcommand{capbtabbox}{table}[][\FBwidth]
\usepackage[T1]{fontenc} 
\usepackage[utf8]{inputenc} 
\usepackage{authblk}
\begin{document}
\title{
  \Large \bf A primal-dual flow for affine constrained convex optimization
}
\author{Hao Luo\thanks{School of Mathematical Sciences, Peking University, Beijing, 100871, China. Email: luohao@math.pku.edu.cn}}
\date{}
\maketitle
\begin{abstract}
We introduce a novel primal-dual flow for affine constrained convex optimization problems. As a modification of the standard saddle-point system, our primal-dual flow is proved to possess the exponential decay property, in terms of a tailored Lyapunov function. Then two primal-dual methods are obtained from  numerical discretizations of the continuous model, and global nonergodic linear convergence rate is established via a discrete Lyapunov function. Instead of solving the subproblem of the primal variable, we apply the semi-smooth Newton iteration to the subproblem with respect to the multiplier, provided that there are some additional properties such as semi-smoothness and sparsity. Especially, numerical tests on the linearly constrained $l_1$-$l_2$ minimization and the total-variation based image denoising model have been provided.
\end{abstract}

\section{Introduction}
\label{sec:intro}
We are interested in the linearly constrained minimization problem
\begin{equation}\label{eq:min-f-Ax-b}
	\mathop{\min}_{x\in\R^n} f(x)\quad {\rm s.t.~} Ax = b,
\end{equation}
where $(A,b)\in\R^{m\times n}\times \R^{m}$ and $f:\R^n\to\R\cup\{+\infty\}$ is  proper, closed and convex.
Let $\Omega: = \R^n\times \R^{m}$ and introduce the Lagrangian 
\begin{equation}\label{eq:L_sg-intro}
	\mathcal L(x,\lambda): =  f(x) + \dual{\lambda,Ax-b}
	\quad \forall\,(x,\lambda)\in \Omega,
\end{equation}
where $\dual{\cdot,\cdot}$ denotes the 
standard $l^2$-inner product, with $\nm{\cdot} = \sqrt{\dual{\cdot,\cdot}}$ being the Euclidean norm. 
Assume $(x^*,\lambda^*)$ is a saddle-point of $\mathcal L(x,\lambda)$, which means 
\[
\mathcal L(x^*,\lambda)\leqslant \mathcal L(x^*,\lambda^*)\leqslant \mathcal L(x,\lambda^*)\quad \forall\,(x,\lambda)\in \Omega,
\]
and denote by $\Omega^*$ the set of all saddle-points.
Any $(x^*,\lambda^*)\in\Omega^*$
satisfies the Karush--Kuhn--Tucker (KKT) system
\begin{equation}\label{eq:kkt}
	\left\{
	\begin{aligned}
		{}&0=\nabla_\lambda\mathcal L(x^*,\lambda^*)=Ax^*-b,\\		
		{}&0\in \partial_x \mathcal L(x^*,\lambda^*)=\partial f(x^*)+A^{\top}\lambda^*,
	\end{aligned}
	\right.
\end{equation}
where $\partial f(x^*)$ is the subdifferential of $f$ at $x^*$.

For the standard model problem \cref{eq:min-f-Ax-b}, there are a large body of primal-dual type algorithms that achieve the fast (ergodic) sublinear rate $\mathcal O(1/k^2)$ with strongly convex condition; see Section \ref{sec:sota} for a brief review. Meanwhile, some existing works also focus on the asymptotic convergence from the continuous-time point of view, i.e., the saddle-point dynamical system \cite{Cherukuri2016,Feijer2010}
\begin{equation}
	\label{eq:sp}
	\left\{
	\begin{aligned}
		{}&\lambda'  ={}\nabla_\lambda \mathcal L(x,\lambda),\\
		{}&		x'  =-\nabla_x \mathcal L(x,\lambda).
	\end{aligned}
	\right.
\end{equation}
In this work, we shall modify the conventional model \cref{eq:sp} and introduce a novel primal-dual flow system which possesses exponential decay property. New primal-dual algorithms shall be obtained from proper time discretizations and nonergodic linear convergence rate will be proved via the tool of Lyapunov function.

To move on, let us make some conventions. We say a function $g:\R^n\to\R$ is $L$-smooth if it has $L$-Lipschitz continuous gradient:
\[
\nm{\nabla g(x)-\nabla g(y)}\leqslant L\nm{x-y}
\quad \forall\,x,y\in\R^n.
\]
For a properly closed convex function $g:\R^n\to\R\cup\{+\infty\}$, it is called $\mu$-convex if there exists $\mu\geqslant0$ such that 
\[
g(x) + \dual{p,y-x}+\frac{\mu}{2}\nm{y-x}^2\leqslant g(y),
\]
for all $p\in\partial g(x)$. The proximal mapping $\prox_{\theta g}:\R^n\to\R^n$ of $g$ with $\theta>0$ is defined by
\[
\prox_{\theta g}(x): =({\rm Id}+\theta\partial g)^{-1}(x)= \mathop{\argmin}\limits_{y\in\R^n}\left\{g(y) + \frac{1}{2\theta}\nm{y-x}^2\right\}\quad\forall\,x\in\R^n,
\]
where $\id$ denotes the identity operator.
Clearly, if $f$ is $\mu$-convex, then according to \cref{eq:L_sg-intro}, we claim that $\mathcal L(\cdot,\lambda)$ is also $\mu$-convex and 
\begin{equation}\label{eq:mu-conv-L}
	\mathcal L(x,\lambda) + \dual{p+A^\top\lambda,y-x}+\frac{\mu}{2}\nm{y-x}^2\leqslant \mathcal L(y,\lambda),\quad p\in\partial f(x).
\end{equation}
\subsection{Main results}
Following the time rescaling technique and the tool of Lyapunov function from \cite{luo_differential_2019,chen_unified_2021}, for smooth and $\mu$-convex objective $f$, we propose a primal-dual flow
\begin{equation}
	\label{eq:pdf-x-intro}
	\left\{
	\begin{aligned}
		{}&			\gamma x'  =-\nabla_x \mathcal L(x,\lambda),\\
		{}&\beta\lambda'  ={}\nabla_\lambda \mathcal L(x+x',\lambda),
	\end{aligned}
	\right.
\end{equation}
where $\gamma$ and $\beta$ are two nonnegative scaling factors that are governed by $\gamma' ={}\mu-\gamma$ and $\beta' =-\beta$, respectively. 
Compared with the classical one \cref{eq:sp}, our new model \cref{eq:pdf-x-intro} has two novelties: (i) it introduces two built-in time rescaling factors that unify the analysis for $\mu\geqslant 0$; (ii) the term $\nabla_\lambda \mathcal L(x+x',\lambda)$ (instead of the standard one $\nabla_\lambda \mathcal L(x,\lambda)$) brings stability and reduces the oscillation; see Section \ref{sec:simple-model} for an illustrative equilibrium analysis.  Besides, the extra term $x'$ in $\nabla_\lambda \mathcal L(x+x',\lambda)$ has subtle connection with the over-relaxation $x_{k+1}+\theta(x_{k+1}-x_k)$ introduced in the primal-dual hybrid gradient (PDHG) method \cite{chambolle_first-order_2011}; see Appendix \ref{app:pdhg} for a more reasonably intrinsic explanation.

We then equip the dynamical system \cref{eq:pdf-x-intro} with a tailored Lyapunov function
\[
\mathcal E(t) := \mathcal L(x(t),\lambda^*) - \mathcal L(x^*,\lambda(t)) + \frac{\gamma(t)}{2}\nm{x(t)-x^*}^2+\frac{\beta(t)}{2}\nm{\lambda(t)-\lambda^*}^2,\quad\,t\geqslant 0,
\]
which possesses the exponential decay (see \cref{lem:pdf-sys-dE-G})
\begin{equation}\label{eq:exp-decay-intro}
	\frac{\dd}{\dd t}\mathcal E(t)\leqslant -\mathcal E(t)\quad\Longrightarrow\quad\mathcal E(t)
	\leqslant e^{-t}\mathcal E(0),\quad t\geqslant 0.
\end{equation}
From \cref{eq:exp-decay-intro} we have $\mathcal L(x(t),\lambda^*)-\mathcal L(x^*,\lambda(t))\leqslant e^{-t}\mathcal E(0)$, and we can further prove  $\snm{f(x(t))-f(x^*)}+\nm{Ax(t)-b}\leqslant Ce^{-t}$; see \cref{coro:conv}.

We also consider implicit and semi-implicit discretizations for the continuous flow \cref{eq:pdf-x-intro} (in general nonsmooth setting) and obtain new primal-dual algorithms, which are close to the (linearized) proximal augmented Lagrangian method but adopt automatically changing parameters. In addition, instead of solving the subproblem of the primal variable, we apply the semi-smooth Newton (SsN) iteration to the subproblem with respect to the multiplier, provided that there are some hidden structures such as semi-smoothness and sparsity. By using a unified discrete Lyapunov function 
\[
\mathcal E_k = \mathcal L(x_k,\lambda^*) - \mathcal L(x^*,\lambda_k) + \frac{\gamma_k}{2}\nm{x_k-x^*}^2+\frac{\beta_k}{2}\nm{\lambda_k-\lambda^*}^2,
\]
we prove the contraction property: 
\[
\mathcal E_{k+1}-\mathcal E_{k}\leqslant -\alpha_k\mathcal E_{k+1}
\quad\text{or}\quad 
\mathcal E_{k+1}-\mathcal E_{k}\leqslant -\alpha_k\mathcal E_{k}\quad\forall\,k\in\mathbb N,
\]
from which we obtain {\it nonergodic} convergence rates of the objective gap  $\snm{f(x_k)-f(x^*)}$ and the feasibility residual $\nm{Ax_k-b}$. More precisely, the implicit discretization converges with (super) linear rate for convex objective $f$ and the semi-implicit scheme possesses the rate $\mathcal O(\min\{L/k,(1+\mu/L)^{-k}\})$ for the composite case $f = h+g$ where $h$ is $L$-smooth and $\mu$-convex and $g$ is convex (possibly nonsmooth).
\subsection{Related works}
\label{sec:sota}
As one can add the indicator function of the constraint set 
to the objective and get rid of the linear constraint in \eqref{eq:min-f-Ax-b},
the proximal gradient method \cite{palomar_gradient-based_2009}, as well as the accelerated proximal gradient method \cite{beck_fast_2009,chen_luo_first_2019,luo_differential_2019,nesterov_gradient_2013}, 
can be considered. However, they need projections onto the affine constraint set and are not suitable to handle the composite case $f = h+g$. 

Therefore, prevailing algorithms 
are the augmented Lagrangian method (ALM) \cite{Bertsekas_2014}, the Bregman iteration \cite{Osher2005} and their variants (linearization or acceleration) \cite{Huang2013,Kang2013,Xu2017,kang_inexact_2015,tao_accelerated_2016,tran-dinh_constrained_2014,tran-dinh_smooth_2018}. Another type of algorithm is the quadratic penalty method with continuation technique \cite{Lan2013,Li2017}. 
Among those methods mentioned here, the fast rate $\mathcal O(1/k^2)$ is mainly in ergodic sense for primal variable and it is rare to see global nonergodic linear rate, even with strongly convex objectives. More recently, Li, Sun and Toh \cite{Li2020d} proposed a (super) linearly convergent semi-smooth Newton based inexact proximal ALM for linear programming. Later, this method has been extended to quadratic programming \cite{Li2020e,Niu2019a}.

For the separable case: $ f(x) = f_1(x_1)+f_2(x_2),\,A = (A_1,A_2)$, we have alternating direction method of multipliers (ADMM) \cite{gabay_dual_1976,Fortin_1983} for primal problem and 
operator splitting methods \cite{douglas_numerical_1956,peaceman_numerical_1955,eckstein_augmented_2012} for dual problem. For ADMM type methods, the sublinear rate $\mathcal O(1/k^2)$ can be proved under partially strong convexity assumption \cite{sabach_faster_2020,tran-dinh_non-stationary_2020,tran-dinh_proximal_2019,tran-dinh_augmented_2018} and global linear rate has been established as well for strongly convex (smooth) objectives \cite{davis_faster_2015,giselsson_linear_2017,deng_global_2016}. In addition, (local) linear convergence can be derived from the error bound condition \cite{aspelmeier_local_2016,han_linear_2015,liu_partial_2018,yang_linear_2016,yuan_discerning_2020}. For a special case $A_1 = I$ or $A_2=I$, there are primal-dual splitting methods \cite{chambolle_introduction_2016,chambolle_first-order_2011,he_convergence_2014,esser_general_2010,pock_algorithm_2009,zhu_ecient_2008,jiang_approximate_2021,bonettini_convergence_2012}. Generally speaking, we have sublinear rate $\mathcal O(1/k^2)$ for partially strongly convex case and linear rate for strongly convex case \cite{chambolle_ergodic_2016,valkonen_inertial_2020,tran-dinh_unified_2021}. Moreover, equivalence between primal-dual splitting methods and ADMM type methods can be found in \cite{clason_nonsmooth_2020,yan_self_2015,oconnor_equivalence_2018}.

On the other hand, ordinary differential equation (ODE) solver approach 
has been revisited nowadays for investigating and developing optimization methods.
For unconstrained problems, there are heavy ball model
\cite{attouch_heavy_2000}, asymptotically vanishing dynamical (AVD) model 
\cite{su_dierential_2016} and their extensions
\cite{attouch_fast_convergence_2018,attouch_fast_2016,wibisono_variational_2016,wilson_lyapunov_2021,lin_control-theoretic_2019}. Besides, Luo and Chen \cite{luo_differential_2019} proposed the so-called Nesterov accelerated gradient flow and later generalized it to \cite{chen_luo_first_2019,chen_unified_2021,luo_accelerated_2021}. 

For linearly constrained problem \cref{eq:min-f-Ax-b}, apart from the classical first-order saddle-point system \cref{eq:sp}, some second-order dynamics have been proposed as well. Zeng, Lei and Chen \cite{Zeng2019} generalized the AVD model
and obtained the decay rate $\mathcal O(t^{-\mathop{\min}\{2,2\alpha/3\}})$
via a suitable Lyapunov function. He, Hu and Fang \cite{He2020} extended the dynamical system in \cite{Zeng2019} to separable case.
Revisiting the scaled alternating direction method of multipliers \cite{Boyd2010}, Franca, Robinson and Vidal \cite{Franca2018b} derived a continuous model which is also related to the AVD model and proved the decay rate $\mathcal O(1/t^2)$. Yet, none of Zeng et al. \cite{Zeng2019}, He et al. \cite{He2020} and Franca et al. \cite{Franca2018b} neither considered numerical discretizations for their dynamical systems nor presented new optimization algorithms for the original optimization problem. For general minimax problems, there are some works on dynamical system approach \cite{Lu2020,cherukuri_saddle-point_2017}.

Comparing with existing works, we summarize our main contributions as below:
\begin{itemize}
	\item  The continuous primal-dual flow \cref{eq:pdf-x-intro} adopts built-in time rescaling factors for both convex and strongly convex cases and has exponential decay rate with respect to a proper Lyapunov function.
	\item A simple but illustrative equilibrium analysis shows the gain of stability that is benefit from the modification introduced in \cref{eq:pdf-x-intro}.
	\item New primal-dual algorithms with automatically changing parameters are obtained from proper time discretizations of the continuous model and the semi-smooth Newton method is considered for the subproblem with respect to the multiplier.
	\item Nonergodic (super) linear convergence rate of the objective gap and feasibility residual is established via the tool of discrete Lyapunov function.
\end{itemize}

The rest of this paper is organized as follows. Section \ref{sec:ode} starts from the classical saddle-point system and introduces a new primal-dual flow. Then Sections \ref{sec:im} and \ref{sec:comp} consider implicit and semi-implicit discretizations respectively and establish the (super) linear convergence rates of the resulted primal-dual algorithms. Numerical performances on the $l_1$-$l_2$ minimization and the total-variation based denoising model are presented in Section \ref{sec:numer} and finally, some concluding remarks are given in Section \ref{sec:conclu}.
\section{Continuous Problems}
\label{sec:ode}
\subsection{The saddle-point system}
\label{sec:try-pdf}
To present the main idea clearly, let us start from the rescaled saddle-point system
\begin{equation}	\label{eq:pdf-try}
	\left\{
	\begin{aligned}
		{}&			\beta\lambda'  ={}\nabla_\lambda \mathcal L(x,\lambda),\\
		{}&	\gamma x'  =-\nabla_x \mathcal L(x,\lambda),
	\end{aligned}
	\right.
\end{equation}
with the initial condition $(x(0),\lambda(0))=(x_0,\lambda_0)\in\Omega$, where $\gamma $ and $\beta$ are two artificial time rescaling factors and satisfy (cf. \cite{luo_differential_2019,chen_unified_2021})
\begin{equation}\label{eq:beta}
	\gamma' =\mu-\gamma\quad \beta' =-\beta,
\end{equation}
with positive initial condition $(\gamma(0),\beta(0))=(\gamma_0,\beta_0)$. 
One can easily solve \cref{eq:beta} to obtain
\begin{equation}\label{eq:sol-beta}
	\gamma(t) = {}\gamma_0 e^{-t}+\mu(1-e^{-t})\quad\text{and}\quad
	\quad \beta(t) = {}\beta_0e^{-t}\quad t\geqslant0,
\end{equation}
which implies $\gamma$ and $\beta$ are positive and converge exponentially to $\mu$ and $0$, respectively. 

Assume $f\in C^1_L$ and define $F:\R_+\times \Omega\to \Omega$ by that
\[
F(t,Z): = \begin{pmatrix}
	\displaystyle
	-	\frac{1}{\gamma(t)}\nabla_x \mathcal L(x,\lambda)\\
	\displaystyle
	\frac{1}{\beta(t)}\nabla_\lambda \mathcal L(x,\lambda)
\end{pmatrix}\quad\forall\, Z = \begin{pmatrix}
	x\\\lambda
\end{pmatrix}\in\Omega.
\]
Then \cref{eq:pdf-try} can be rewritten as $Z'(t) = F(t,Z(t))$ and a direct calculation yields that for all $Z,\,Y\in \Omega$ and $0\leqslant s\leqslant t$,
\[
\nm{F(t,Z)-F(s,Y)}\leqslant C_0(L+\nm{A})\left(
\snm{t-s}\nm{Z-Z^*} + \nm{Z-Y}
\right)e^t,
\]
where $Z^*=(x^*,\lambda^*)\in\Omega^*$ and the bounded positive constant $C_0$ depends only on $\gamma_0,\beta_0$ and $\mu$. This means $F$ is locally Lipschitz continuous and according to \cite[Proposition 6.2.1]{haraux_1991} and \cite[Corollary A.2]{brezis_1973}, the first-order dynamical system \cref{eq:pdf-try} exists a unique solution $Z = (x,\lambda)\in C^1(\R_+;\Omega)$. 

Let $\boldsymbol V := \Omega\times\R_{+}\times\R_{+}$ and for any $X = (x,\lambda,\gamma,\beta)\in \boldsymbol V$,
introduce a Lyapunov function
\begin{equation}\label{eq:EX}
	\mathcal E(X): = \mathcal L(x,\lambda^*)-\mathcal L(x^*,\lambda)
	+\frac{\gamma}{2}\nm{x-x^*}^2
	+\frac{\beta}{2}\nm{\lambda-\lambda^*}^2.
\end{equation}
Our goal is to establish the exponential decay property of \eqref{eq:EX} along with the solution trajectory $X:\R_+\to\boldsymbol  V$. Below, we present a lemma which violates our goal but heuristically motivates us to the right way.
\begin{lem}\label{lem:pdf-try-dE-G}
	Assume $f$ is $L$-smooth and $\mu$-convex with $\mu\geqslant 0$ and let $X=(x,\lambda,\gamma,\beta):\R_+\to\boldsymbol V$ be the unique solution to \cref{eq:pdf-try,eq:beta}, then 
	\begin{equation}\label{eq:pdf-try-dE-G}
		\frac{\rm d }{{\rm d} t}	
		\mathcal E(X)
		\leqslant -\mathcal E(X)
		-\gamma\nm{x'}^2-\dual{Ax',\lambda-\lambda^*}.
	\end{equation}
\end{lem}
\begin{proof}
	As discussed above, $(x,\lambda)\in C^1(\R_+;\Omega)$ exists uniquely and by \cref{eq:pdf-try}, a direct computation gives
	\[
	\begin{split}
		\frac{\rm d }{{\rm d} t}	
		\mathcal E(X)
		={}&\dual{\nabla_x\mathcal E(X),x'}+\dual{\nabla_\lambda\mathcal E(X),\lambda'}+
		\dual{\nabla_\gamma\mathcal E(X),\gamma'}	+\dual{\nabla_\beta\mathcal E(X),\beta'}\\
		={}&\underbrace{	-\frac{1}{\gamma}\dual{\nabla_x\mathcal L(x,\lambda), \nabla_x\mathcal L(x,\lambda^*)}}_{I_1}+\underbrace{\dual{\nabla_\lambda \mathcal L(x,\lambda),\lambda-\lambda^*}
			-\dual{\nabla_x\mathcal L(x,\lambda),x-x^*}}_{I_2}\\
		{}&\qquad\underbrace{-\frac{\beta}{2}\nm{\lambda-\lambda^*}^2
			+ \frac{\mu-\gamma}{2}\nm{x-x^*}^2}_{I_3}\\
		:={}&I_1+I_2+I_3.
	\end{split}
	\]
	We split $\nabla_x\mathcal L(x,\lambda^*)=\nabla_x\mathcal L(x,\lambda)-A^{\top}(\lambda-\lambda^*)$ and use the relation $\gamma x' = -\nabla_x\mathcal L(x,\lambda)$ to get
	\begin{equation}\label{eq:I1}
		\begin{split}
			I_1=&-	\frac{1}{\gamma}\dual{\nabla_x\mathcal L(x,\lambda), \nabla_x\mathcal L(x,\lambda)-A^{\top}(\lambda-\lambda^*)} 
			=-\gamma\nm{x'}^2-\dual{Ax',\lambda-\lambda^*}.
		\end{split}
	\end{equation}
	Also, we reformulate $I_2$ as follows
	\begin{equation}\label{eq:I2}\small
		\begin{split}
			I_2
			={}&
			\dual{\nabla_\lambda \mathcal L(x,\lambda),\lambda-\lambda^*}
			-\dual{Ax-Ax^*,\lambda-\lambda^*}
			-\dual{\nabla_x\mathcal L(x,\lambda^*),x-x^*}
			=-\dual{\nabla_x\mathcal L(x,\lambda^*),x-x^*},
		\end{split}
	\end{equation}
	where we have used the optimality condition $Ax^*=b$.
	Since $f$ is $\mu$-convex, we know that $\mathcal L(\cdot,\lambda^*)$ is also $\mu$-convex and it follows from \cref{eq:mu-conv-L} that
	\[
	\begin{split}
		I_2
		\leqslant {}&\mathcal L(x,\lambda^*)-\mathcal L(x^*,\lambda^*)-\frac{\mu}{2}\nm{x-x^*}^2
		={}\mathcal L(x,\lambda^*)-\mathcal L(x^*,\lambda)-\frac{\mu}{2}\nm{x-x^*}^2.
	\end{split}
	\]
	Here, recall the fact that $\mathcal L(x^*,\cdot)$ is a constant. Hence, collecting $I_3$, \cref{eq:I1,eq:I2} proves \eqref{eq:pdf-try-dE-G}.
\end{proof}
To obtain $\mathcal E'(X)\leqslant -\mathcal E(X)$ from  \eqref{eq:pdf-try-dE-G}, we shall prove $-\gamma\nm{x'}^2-\dual{Ax',\lambda-\lambda^*}\leqslant 0$.
In stead of twisting on the existence of this, in the next section, we resort to introducing a subtle modification that cancels exactly the cross term $\dual{Ax',\lambda-\lambda^*}$ in \eqref{eq:pdf-try-dE-G} and finally leads to the desired estimate.
\subsection{A new primal-dual flow}
\label{sec:pd-flow}
Although \eqref{eq:pdf-try-dE-G} fails to give the desired result, it suggests a simple remedy: replacing $\nabla _\lambda \mathcal L(x,\lambda)$ 
by $\nabla _\lambda \mathcal L(x+x',\lambda)$. Then the first part $I_1$ (cf. \eqref{eq:I1}) 
brings one more term $\dual{Ax',\lambda-\lambda^*}$ which offsets exactly the last term in \eqref{eq:pdf-try-dE-G} while both $I_2$ and $I_3$ keep unchanged. 

Namely, we leave the parameter system \cref{eq:beta} invariant but modify \cref{eq:pdf-try} properly to obtain a novel primal-dual flow
\begin{subnumcases}{}
	\gamma x'  =-\nabla_x \mathcal L(x,\lambda),
	\label{eq:pdf-x}\\
	\beta\lambda'  ={}\nabla_\lambda \mathcal L(x+x',\lambda).
	\label{eq:pdf-lambda}
\end{subnumcases}
Similar with \cref{eq:pdf-try}, we claim that \cref{eq:pdf-lambda} admits a unique classical solution $(x,\lambda)\in C^1(\R_+;\Omega)$. We also mention that the extrapolation idea $x+x'$ in \eqref{eq:pdf-lambda} can be found previously in the second-order primal-dual ODE proposed by \cite{Zeng2019}. In the sequel, we shall complete the exponential decay of the Lyapunov function \eqref{eq:EX} and then provide an illustrative equilibrium analysis that gives a convincible explanation of the subtle modification $x+x'$. Additionally, in Appendix \ref{app:pdhg}, we present an over-relaxation perspective, which perhaps shows the intrinsic connection with the PDHG method \cite{chambolle_first-order_2011}.
\begin{thm}\label{lem:pdf-sys-dE-G} 
	Assume $f$ is $L$-smooth and $\mu$-convex with $\mu\geqslant 0$ and let $X=(x,\lambda,\gamma,\beta):\R_+\to\boldsymbol V$ be the unique solution to \cref{eq:beta,eq:pdf-x}, then 
	\begin{equation}\label{eq:pdf-sys-EG}
		\frac{\rm d }{{\rm d} t}	\mathcal E(X)
		\leqslant -\mathcal E(X)-\gamma\nm{x'}^2.
	\end{equation}
	Consequently, we have the exponential decay
	\begin{equation}\label{eq:exp-rate}
		\mathcal E(X(t))+\int_{0}^{t}e^{s-t}\gamma(s)\nm{x'(s)}^2{\rm d}s
		\leqslant e^{-t}	\mathcal E(X(0)),
		\quad 0\leqslant t<\infty.
	\end{equation}
\end{thm}
\begin{proof}
	According to the above discussions, the proof of \eqref{eq:pdf-sys-EG} is in line with that of \eqref{eq:pdf-try-dE-G} and thus omitted here. The estimate \eqref{eq:exp-rate} follows from \eqref{eq:pdf-sys-EG} immediately.
\end{proof}

Thanks to the two scaling factors introduced
in \cref{eq:beta}, the exponential decay \eqref{eq:exp-rate} holds 
uniformly for $\mu\geqslant0$. Let $\gamma_{\min}:=\min\{\gamma_0,\mu\} $, then by \cref{eq:sol-beta}, we have 
\begin{equation}\label{eq:bd-gama}
	\gamma(t)\geqslant \max\left\{
	\gamma_{\min},\,\gamma_{0} e^{-t}
	\right\}\quad\forall\,t\geqslant 0.
\end{equation}
Furthermore, we have a corollary which gives: (i) the boundness of $\lambda(t)$ and $x(t)$; (ii) exponential decay of the Lagrangian $\mathcal L(x(t),\lambda^*)-\mathcal L(x^*,\lambda(t))$, the primal objective residual $\snm{f(x(t))-f(x^*)}$ and the feasibility violation $\nm{Ax(t)-b}$; (iv) the integrability of $\nm{x'(t)}$.
\begin{coro}\label{coro:conv}
	Assume $f$ is $L$-smooth and $\mu$-convex with $\mu\geqslant 0$. Then for the unique solution $(x,\lambda):\R_+\to\Omega$ of \cref{eq:pdf-x} , we have the following.
	\begin{enumerate}
		\item $\sqrt{\gamma_{0}+\gamma_{\min}e^t}\nm{x'(t)}\in L^2(0,\infty)$.			
		\item $0\leqslant\mathcal L(x(t),\lambda^*)-\mathcal L(x^*,\lambda(t))\leqslant e^{-t}\mathcal E(X(0))$.						
		\item $\lambda(t)$ is bounded: $\beta_0\nm{\lambda(t)-\lambda^*}^2\leqslant 2\mathcal E(X(0))$.
		\item $x(t)$ is bounded: $\gamma_0\nm{x(t)-x^*}^2\leqslant 2\mathcal E(X(0))$ and $\gamma_{\min}\nm{x(t)-x^*}^2
		\leqslant 2e^{-t}\mathcal E(X(0))$.					
		\item $\nm{Ax(t)-b}\leqslant e^{-t}\mathcal R_0$ and $\snm{f(x(t))-f(x^*)}\leqslant e^{-t}\big(\mathcal E(X(0))+\mathcal R_0\nm{\lambda^*}\big)$, where 
		\[
		\mathcal R_0:=\sqrt{2\beta_0\mathcal E(X(0))}+\beta_0\nm{\lambda_0-\lambda^*}+\nm{Ax_0-b}.
		\]
	\end{enumerate}
\end{coro}
\begin{proof}
	The first to the fourth follow directly from \eqref{eq:EX}, \eqref{eq:exp-rate} and \eqref{eq:bd-gama}. Let us prove the last one. 
	Define $\xi(t): = \lambda(t) - \beta^{-1}(t)(Ax(t)-b)$,
	then by \eqref{eq:beta} and \eqref{eq:pdf-lambda}, 
	\begin{equation}\label{eq:dr/dt}
		\frac{{\rm d}\xi}{{\rm d} t} = \lambda'(t)-\beta^{-1}(t)\left(Ax'(t)+Ax(t)-b\right)=0,
	\end{equation}
	which says $\xi(t)=\xi(0)$ and also implies that
	\[
	\nm{Ax(t)-b}
	=\beta(t)\nm{\lambda(t)-\xi(0)}
	\leqslant \beta(t)\big(\nm{\lambda(t)-\lambda^*}+\nm{\xi(0)-\lambda^*}\big).
	\]
	Hence, from the fact $\beta(t) = \beta_0e^{-t}$ 
	and the boundness of $\nm{\lambda(t)-\lambda^*}$, we have
	\begin{equation}\label{eq:feas-res}
		\nm{Ax(t)-b}
		\leqslant 	e^{-t}\left(\sqrt{2\beta_0\mathcal E(X(0))}+\beta_0\nm{\xi(0)-\lambda^*}\right)
		\leqslant 
		e^{-t}\mathcal R_0.
	\end{equation}
	Besides, it follows from \cref{eq:exp-rate} that
	\[
	0\leqslant 	\mathcal L(x(t),\lambda^*)-\mathcal L(x^*,\lambda(t))=
	f(x(t))-f(x^*)+\dual{\lambda^*,Ax(t)-b}
	\leqslant e^{-t}\mathcal E(X(0)),
	\]
	which together with the previous estimate \cref{eq:feas-res} gives
	\[
	\begin{split}
		\snm{f(x(t))-f(x^*)}\leqslant{}& \nm{\lambda^*}\nm{Ax(t)-b}+
		e^{-t}\mathcal E(X(0))
		\leqslant {}e^{-t}\big(\mathcal E(X(0))+\mathcal R_0\nm{\lambda^*}\big).
	\end{split}
	\]
	This establishes the exponential decay of the primal objective error and completes the proof.
\end{proof}
\begin{rem}
	From \cref{coro:conv}, we conclude that for $\mu\geqslant 0$, the primal-dual gap $\mathcal L(x(t),\lambda^*)-\mathcal L(x^*,\lambda(t))$, the primal objective residual $\snm{f(x(t))-f(x^*)}$ and the feasibility violation $\nm{Ax(t)-b}$ decrease exponentially. We also have strong convergence: $\nm{x(t)-x^*}^2\leqslant Ce^{-t}$ for the strongly convex case $\mu>0$. \qed
\end{rem}
\begin{rem} \label{rem:nonsmooth}
	We mention that the well-posedness of \cref{eq:pdf-lambda} with general nonsmooth objective $f$ is of interest to study further. 
	As we can see, the modified system \cref{eq:pdf-lambda} promises the exponential decay \cref{eq:pdf-sys-EG} but it is totally different from the original one \cref{eq:pdf-try}. In nonsmooth setting, \cref{eq:pdf-try} can be almost viewed as a dynamical system governed by a maximally monotone operator:
	\begin{equation}\label{eq:Z'}
		Z'(t)+\Lambda(t)M(Z(t))\ni0,
	\end{equation}
	where $\Lambda(t) = {\rm diag}(\gamma^{-1}(t)I_n,\beta^{-1}(t)I_m),\,Z(t) = (x(t),\lambda(t))$ and  the maximally monotone operator $M:\Omega\to 2^{\Omega}$ is defined by that
	\begin{equation}\label{eq:M}
		M(Z): = \begin{pmatrix}
			\partial f(x) + A^\top\lambda\\
			b-	Ax
		\end{pmatrix}\quad\forall\,Z=\begin{pmatrix}
			x\\\lambda
		\end{pmatrix}\in\Omega.
	\end{equation}
	According to \cite[Section 4.2]{djafari-rouhani_nonlinear_2019}, we claim that \cref{eq:Z'} admits a unique solution $Z=(x,\lambda)\in W^{1,\infty}_{\rm loc}(\R_+;\Omega)$. However, our primal-dual flow \cref{eq:pdf-lambda} reads as (cf. \cref{eq:pdf-x-non})
	\begin{equation}\label{eq:Z-RM}
		Z'(t)+R(t)M(Z(t))\ni0,
	\end{equation}
	where $R(t)$ is a lower triangular matrix:
	\[
	R(t) = \begin{pmatrix}
		\gamma^{-1}(t)\id&O\\
		\gamma(t)^{-1}\beta^{-1}(t)A&	\beta^{-1}(t)\id
	\end{pmatrix}.
	\]	
	The existence and uniqueness of the solution to \cref{eq:Z-RM} is under studying. In addition, both the exponential decay \cref{eq:pdf-sys-EG} and (weak) convergence of the trajectory $Z(t)$ to a saddle-point $(x_\infty,\lambda_\infty)\in\Omega^*$ deserve future investigations. \qed
\end{rem}
\subsection{A simple equilibrium analysis}
\label{sec:simple-model}
Let $p>2$ be a positive even integer and consider a simple smooth convex function
\[
f(x) = \frac{1}{p}\left(x_1^p+x_2^p\right)
\quad\forall\, x = \begin{pmatrix}
	x_1\\x_2
\end{pmatrix}\in\R^2,
\]
with the linear constraint $ax = x_1-x_2 = 0$, where $a = (1,-1)$.
Clearly $(x^*,\lambda^*)=(0,0,0)$ is the unique saddle point.
Take $\mu = 0$ and for simplicity we choose $\gamma_0=\beta_0=1$, then $\gamma(t)=\beta(t)=e^{-t}$ and the original model \cref{eq:pdf-try} 
becomes
\begin{equation}\label{eq:pdf-try-simple}
	\left\{
	\begin{split}
		\lambda'= &\,e^t(x_1-x_2),\\
		x_1'= &-e^t(x_1^{p-1}+\lambda),\\
		x_2' = &-e^t(x_2^{p-1}-\lambda).
	\end{split}
	\right.
\end{equation}
The ``linearization" around $(x^*,\lambda^*)$ is
\[
\begin{pmatrix}
	\widehat{\lambda}\\	\widehat x
\end{pmatrix}'
=e^{t}B
\begin{pmatrix}
	\widehat{\lambda}\\	\widehat x
\end{pmatrix}\quad\text{with}\quad B = 
\begin{pmatrix}
	0&a\\-a^\top &O
\end{pmatrix}.
\]
Note that $B$ has three distinct eigenvalues: $b_1=0,\,b_2=-i\sqrt{2}$ and $b_3=i\sqrt2$. 
This implies $(x^*,\lambda^*)$ is stable but not asymptotically stable and the solution trajectory of \cref{eq:pdf-try-simple} will spin around $(x^*,\lambda^*)$ with high oscillation 
and thus converges dramatically slowly. 

The modified system \cref{eq:pdf-x} reads as follows
\begin{equation}\label{eq:pdf-simple}
	\left\{
	\begin{split}
		\lambda' = &\,e^t(x_1+x_1'-x_2-x_2'),\\
		x_1'= &-e^t(x_1^{p-1}+\lambda),\\
		x_2' = &-e^t(x_2^{p-1}-\lambda),
	\end{split}
	\right.
\end{equation}
and its ``linearization" at $(x^*,\lambda^*)$ is 
\[
\begin{pmatrix}
	\widehat{\lambda}\\	\widehat x
\end{pmatrix}'
=e^{t}\widehat{B}(t)
\begin{pmatrix}
	\widehat{\lambda}\\	\widehat x
\end{pmatrix}\quad\text{with}\quad \widehat{B}(t) = 
\begin{pmatrix}
	-2e^t&a\\-a^\top &O
\end{pmatrix}.
\]
Given any fixed time $t\geqslant \ln\sqrt{2}$, all the eigenvalues of $\widehat{B}(t)$ are
\[
\widehat b_1= 0,\quad
\widehat b_2=-\frac{2}{e^t+\sqrt{e^{2t}-2}}\quad\text{and}\quad
\widehat b_3=-e^t-\sqrt{e^{2t}-2}.
\]
From this, we observe more negativity of the real part of nonzero eigenvalues and hopefully the solution $(x,\lambda)$ to the modified system \cref{eq:pdf-simple} 
converges to $(x^*,\lambda^*)$ more quickly. 

In conclusion, our primal-dual flow \cref{eq:pdf-lambda} with subtle extrapolation $x+x'$ reduces the oscillation and accelerates the convergence; see Figure \ref{fig:p-6}. 
\begin{figure}[H]
	\centering
		\includegraphics[width=384pt,height=106pt]{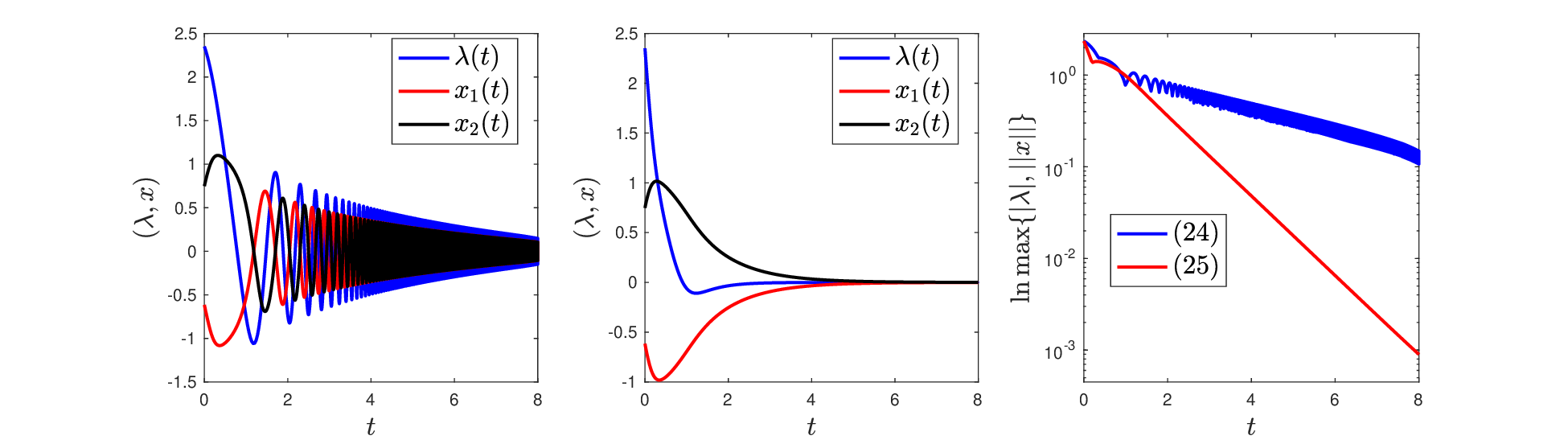}
	\caption{Solution trajectories (the left for \cref{eq:pdf-try-simple} and the medium for \cref{eq:pdf-simple}) and their errors (the right) with $p = 6$.}
	\label{fig:p-6}
\end{figure}
\section{An Implicit Scheme}
\label{sec:im}
From now on, we move to discrete level and consider general nonsmooth $\mu$-convex objective $f$ with $\mu\geqslant 0$. In this setting our primal-dual flow \cref{eq:pdf-x} becomes a differential inclusion
\begin{subnumcases}{}
	\gamma x'  \in-\partial_x \mathcal L(x,\lambda),
	\label{eq:pdf-x-non}\\
	\beta\lambda'  ={}\nabla_\lambda \mathcal L(x+x',\lambda),
	\label{eq:pdf-lambda-non}
\end{subnumcases}
where $\partial_x \mathcal L(x,\lambda) = \partial f(x) + A^\top\lambda$.
As discussed in \cref{rem:nonsmooth}, well-posedness of the solution to \cref{eq:pdf-lambda-non} in proper sense is left as a future topic. In what follows, we shall present new primal-dual algorithms  based on implicit Euler discretization (this section) and semi-implicit discretization (the next section), respectively. Similar with the continuous level, the tool of Lyapunov function plays important role in convergence rate analysis.
\subsection{Implicit discretization}
\label{sec:im-algo}
We first consider an implicit Euler scheme for \cref{eq:pdf-x-non}:
\begin{subnumcases}{}
	{}v_{k+1} = x_{k+1}+\frac{x_{k+1}-x_{k}}{\alpha_k},
	\label{eq:pdf-sys-im-v}\\
	\beta_k \frac{\lambda_{k+1}-\lambda_k}{\alpha_k} = {} \nabla_\lambda \mathcal L(v_{k+1},\lambda_{k+1}),		
	\label{eq:pdf-sys-im-lambda}\\
	\gamma_k \frac{x_{k+1}-x_k}{\alpha_k} \in- \partial_x\mathcal L(x_{k+1},\lambda_{k+1}),
	\label{eq:pdf-sys-im-x}
\end{subnumcases}
where $\alpha_k>0$ denotes the step size and the parameter system \cref{eq:beta} is also discretized implicitly 
\begin{equation}\label{eq:betak-im}
	\frac{\gamma_{k+1}-\gamma_{k}}{\alpha_k}={}\mu-\gamma_{k+1},\quad
	\frac{\beta_{k+1}-\beta_{k}}{\alpha_k}=-\beta_{k+1}.
\end{equation}

Let us transform the time discretization \cref{eq:pdf-sys-im-x} to a primal-dual algorithm. 
From \eqref{eq:pdf-sys-im-x} it follows that
\begin{equation}\label{eq:xk1}
	\frac{x_{k+1}-x_k+\theta_kA^\top\lambda_{k+1}}{\theta_k}\in-\partial f(x_{k+1}),
\end{equation}
where $\theta_k = \alpha_k/\gamma_k$.
Plugging \eqref{eq:pdf-sys-im-v} into \eqref{eq:pdf-sys-im-lambda} and using \cref{eq:betak-im}, we find
\begin{equation}\label{eq:lk1}
	\lambda_{k+1} =\lambda_k-\frac{1}{\beta_k}(Ax_k-b)
	+ \frac{1}{\beta_{k+1}}(Ax_{k+1}-b).
\end{equation}
Then, combining \cref{eq:xk1,eq:lk1} gives
\[
\frac{x_{k+1}-\widehat{x}_k}{\theta_k}+\frac{1}{\beta_{k+1}}A^\top(Ax_{k+1}-b)+A^\top\lambda_k\in-\partial f(x_{k+1}),
\]
where $\widehat{x}_k :=x_k
+\theta_k/\beta_kA^{\top}(Ax_k-b)$. Consequently, we obtain
\begin{subnumcases}{}
	x_{k+1}
	={}\mathop{\argmin}_{x\in\R^n}
	\left\{
	\mathcal L(x,\lambda_{k})+
	\frac{1}{2\beta_{k+1}}\nm{Ax-b}^2+
	\frac{1}{2\theta_k}\nm{x-\widehat{x}_k}^2
	\right\},
	\label{eq:pdf-sys-im-xk1-argmin}\\
	v_{k+1} = {}x_{k+1}+(x_{k+1}-x_k)/\alpha_k,\\
	\lambda_{k+1}={}\lambda_k + \alpha_k/\beta_k(Av_{k+1}-b).
	\label{eq:pdf-sys-im-lk1-argmin}
\end{subnumcases}
Note that in \cref{eq:pdf-sys-im-x} we used only the Lagrangian function $\mathcal L(x,\lambda)$ without the augmented term $\nm{Ax-b}^2$. But in \eqref{eq:pdf-sys-im-xk1-argmin}, the augmented term arises because $\lambda_{k+1}$ and $x_{k+1}$ are coupled with each other in the implicit discretization \cref{eq:pdf-sys-im-x}.

The method \cref{eq:pdf-sys-im-xk1-argmin} is very close the the proximal ALM and the key is to solve the subproblem \eqref{eq:pdf-sys-im-xk1-argmin} with respect to the primal variable $x_{k+1}$. On the other hand, from \cref{eq:xk1} we observe that $x_{k+1} = \proxi_{\theta_k f}(x_k-\theta_kA^{\top}\lambda_{k+1})$.
Putting this back to \cref{eq:pdf-sys-im-lk1-argmin} gives a nonlinear equation in terms of the multiplier $\lambda_{k+1}$:
\begin{equation}\label{eq:lk1-im-nonlinear-eq}
	\beta_{k+1}\lambda_{k+1} -A\proxi_{\theta_k f}\left(x_k-\theta_kA^{\top}\lambda_{k+1}\right)- z_k=0,
\end{equation}
where $z_k = \beta_{k+1}\left(\lambda_k-\beta_{k}^{-1}(Ax_k-b)\right)-b$. As discussed later in Section \ref{sec:semi-comp}, instead of computing $x_{k+1}$ from \eqref{eq:pdf-sys-im-xk1-argmin}, we apply the semi-smooth Newton method \cite{Facchinei2006} to  \cref{eq:lk1-im-nonlinear-eq} to obtain $\lambda_{k+1}$ and then update $x_{k+1}$.  
\begin{rem}
	\label{rem:compare-solver}
	For a better understanding of \cref{eq:pdf-sys-im-xk1-argmin} and \cref{eq:lk1-im-nonlinear-eq} , we give an operator perspective.
	Notice that \cref{eq:pdf-sys-im-v} is a nonlinear saddle-point type equation with respect to $x_{k+1}$ and $\lambda_{k+1}$:
	\begin{equation}\label{eq:A}
		\mathcal A
		\begin{pmatrix}
			x_{k+1}\\\lambda_{k+1}
		\end{pmatrix}
		= r_k\quad\text{where}\quad 
		\mathcal A = 
		\begin{pmatrix}
			\id+	\theta_k\partial f& A^\top\\
			-A&\beta_{k+1}\id
		\end{pmatrix}.
	\end{equation}
	Formally, we have the following  factorizations:
	\[
	\begin{aligned}
		\mathcal A = {}&
		\begin{pmatrix}
			\id&A^\top/\beta_{k+1}\\O&\id
		\end{pmatrix}
		\begin{pmatrix}
			\mathcal P&O\\O&\beta_{k+1}\id
		\end{pmatrix}
		\begin{pmatrix}
			\id&O\\-A/\beta_{k+1}&\id
		\end{pmatrix},\\
		\mathcal A	={}& 
		\begin{pmatrix}
			\id&O\\(\id+\theta_k\partial f)^{-1}(-A)&\id
		\end{pmatrix}
		\begin{pmatrix}
			\id+\theta_k\partial f&O\\O& \mathcal S
		\end{pmatrix}
		\begin{pmatrix}
			\id&(\id+\theta_k\partial f)^{-1}(-A^\top)\\O&\id
		\end{pmatrix},
	\end{aligned}
	\]
	where $\mathcal P = \id+\theta_k\partial f+A^\top A/\beta_{k+1}$ and 
	\[
	\mathcal S = \beta_{k+1}\id-A(\id+\theta_k\partial f)^{-1} (-A^\top) = \beta_{k+1}\id-A\prox_{\theta_kf} (-A^\top) 
	\]
	is nothing but the Schur complement. Hence, to solve \cref{eq:A}, we can compute 
	\[
	\mathcal P^{-1} = 	 \left(\id+\theta_k\partial f+A^\top A/\beta_{k+1}\right)^{-1},
	\]
	which corresponds to the augmented Lagrangian method \cref{eq:pdf-sys-im-xk1-argmin}. On the other hand, one can calculate
	\[
	\mathcal S^{-1} =\left( \beta_{k+1} \id-A\prox_{\theta_kf} (-A^\top) \right)^{-1},
	\]
	which is equivalent to solve the nonlinear equation \cref{eq:lk1-im-nonlinear-eq}. \qed
\end{rem}
Below the implicit scheme \cref{eq:pdf-sys-im-x} (i.e., the method \cref{eq:pdf-sys-im-xk1-argmin}) has been rewritten as an algorithm framework, which is called the {\it implicit primal-dual} (Im-PD) method. 
According to \cref{thm:conv-im} below, we have global linear rate $(1+\widehat{\alpha})^{-k}$ as long as the step size is bounded below $\alpha_k\geqslant\widehat{\alpha}>0$, and superlinear convergence follows if $\alpha_k\to\infty$. Note that this holds even for convex case $\mu=0$. In fact, the fully implicit scheme \cref{eq:pdf-sys-im-lambda} inherits the exponential decay \cref{eq:pdf-sys-EG} from the continuous level, and thus we have the contraction \cref{eq:diff-Ek-im} which has no restriction on the step size $\alpha_k$. Besides, the strong convexity constant $\mu$ of the objective $f$ is not necessarily needed since one can set $\mu=0$ in \cref{eq:pdf-sys-im-lambda} and this  leaves the final rate in \cref{thm:conv-im} unchanged.
\begin{algorithm}[H]
	\caption{Im-PD method for problem \cref{eq:min-f-Ax-b} with $f$ being $\mu$-convex ($\mu\geqslant 0$)}
	\label{algo:Im-PD}
	\begin{algorithmic}[1] 
		\REQUIRE  $\gamma_0>0,\,\beta_0>0,\, x_0
		\in\R^n,\,\lambda_0\in\R^m$.
		\FOR{$k=0,1,\ldots$}
		\STATE Choose the step size $\alpha_k>0$.
		\STATE Update $\displaystyle \beta_{k+1}= \beta_k/(1+\alpha_k)$ and $\displaystyle \gamma_{k+1} =(\mu\alpha_k+\gamma_k)/(1+\alpha_k)$.
		\STATE Set $\displaystyle 		\theta_{k}={}\alpha_k/\gamma_{k}$ and $\displaystyle z_k  ={} \beta_{k+1}\left(\lambda_k-\beta_k^{-1}(Ax_k-b)\right)-b$.
		\STATE Solve $\lambda_{k+1}$ from \cref{eq:lk1-im-nonlinear-eq}
		via the SsN iteration \eqref{eq:SsN} with the line search procedure \eqref{eq:line-search}.
		\STATE Update $x_{k+1}	 = \prox_{\theta_k f}\left(x_k-\theta_k A^{\top}\lambda_{k+1}\right)$.
		\ENDFOR
	\end{algorithmic}
\end{algorithm}
\subsection{Convergence rate}
We now prove the convergence rate of the implicit scheme \cref{eq:pdf-sys-im-lambda} (i.e. Algorithm \ref{algo:Im-PD}) via a discrete analogue to \cref{eq:EX}:
\begin{equation}\label{eq:Ek}
	\mathcal E_k:=
	\mathcal L(x_k,\lambda^*)-\mathcal L(x^*,\lambda_k)
	+\frac{\beta_k}{2}\nm{\lambda_k-\lambda^*}^2
	+\frac{\gamma_k}{2}\nm{x_k-x^*}^2,
\end{equation}
where $(x^*,\lambda^*)\in\Omega^*$ and $\{(x_k,\lambda_k,\gamma_k,\beta_k)\}\in\boldsymbol V$. 
\begin{lem}\label{thm:conv-im}
	Assume $f$ is $\mu$-convex with $\mu\geqslant 0$. Let $\{(x_k,\lambda_k,\gamma_k,\beta_k)\}$ be generated by  Algorithm \ref{algo:Im-PD} with arbitrary step size $\alpha_k>0$, then we have the contraction
	\begin{equation}\label{eq:diff-Ek-im}
		\mathcal E_{k+1}-	\mathcal E_k
		\leqslant -\alpha_k		\mathcal E_{k+1},
		\quad\text{for all }\,k\in\mathbb N.
	\end{equation}
	Moreover, there holds that
	\begin{numcases}{}
		{}\nm{Ax_{k}-b}\leqslant\mathcal R_0\times  \prod_{i=0}^{k-1}\frac{1}{1+\alpha_i},\label{eq:conv-Axk-b-im}\\
		0\leqslant \mathcal L(x_{k},\lambda^*)-	\mathcal L(x^*,\lambda_{k})\leqslant\mathcal E_0\times \prod_{i=0}^{k-1}\frac{1}{1+\alpha_i},
		\label{eq:conv-Lk-im}
		\\
		{}\snm{f(x_k)-f(x^*)}
		\leqslant \left(\mathcal E_0+\mathcal R_0\nm{\lambda^*}\right)\times \prod_{i=0}^{k-1}\frac{1}{1+\alpha_i},
		\label{eq:conv-fk-im}
	\end{numcases}
	where $\mathcal R_0:=
	\sqrt{2\beta_{0}\mathcal E_0}+\beta_{0}
	\nm{\lambda_0-\lambda^*}+\nm{Ax_0-b}$.
\end{lem}
\begin{proof}
	To prove \cref{eq:diff-Ek-im}, we mimic the continuous level (cf. Section \ref{sec:ode})
	but replace the derivative with the difference $		\mathcal E_{k+1}-\mathcal E_k = I_1+I_2+I_3$, where
	\begin{equation}\label{eq:Ii}
		\left\{
		\begin{split}
			I_1:={}&\mathcal L(x_{k+1},\lambda^*)-\mathcal L(x_{k},\lambda^*),\\
			I_2:=		{}&\frac{\beta_{k+1}}{2}
			\nm{\lambda_{k+1}-\lambda^*}^2 - 
			\frac{\beta_{k}}{2}
			\nm{\lambda_{k}-\lambda^*}^2,\\
			I_3:= 		{}&\frac{\gamma_{k+1}}{2}
			\nm{x_{k+1}-x^*}^2 - 
			\frac{\gamma_{k}}{2}
			\nm{x_{k}-x^*}^2.
		\end{split}
		\right.
	\end{equation}
	Let $p_{k+1}=	(x_{k+1}-(x_k-\theta_kA^\top\lambda_{k+1}))/\theta_k$, then by \cref{eq:xk1}, we have $A^\top\lambda^*-p_{k+1}\in\partial \mathcal L(x_{k+1},\lambda^*)$.
	Since $\mathcal L(\cdot,\lambda^*)$ is $\mu$-convex, by \cref{eq:mu-conv-L} we have
	\[
	I_1 
	\leqslant{} \dual{A^\top\lambda^*-p_{k+1},x_{k+1}-x_k}-\frac{\mu}{2}
	\nm{x_{k+1}-x_k}^2.
	\]
	Shift $\lambda^*$ to $\lambda_{k+1}$ and use the relation 
	\begin{equation}\label{eq:pk1}
		p_{k+1}-A^\top\lambda_{k+1} =x_{k+1}-x_k
	\end{equation}
	to lighten the previous estimate as follows
	\begin{equation}\label{eq:I1im}
		I_1 
		\leqslant
		-	 \dual{Ax_{k+1}-Ax_k,\lambda_{k+1}-\lambda^*},
	\end{equation}
	where we dropped the surplus negative term $-\nm{x_{k+1}-x_k}^2$.
	
	Then we focus on $I_2$ and $I_3$. By \cref{eq:betak-im}, a direct computation yields
	\[
	\begin{split}
		I_2
		={}&		-\frac{\alpha_k\beta_{k+1}}{2}\nm{\lambda_{k+1}-\lambda^*}^2
		+\frac{\beta_{k}}{2}
		\left(\nm{\lambda_{k+1}-\lambda^*}^2 - 
		\nm{\lambda_{k}-\lambda^*}^2\right)\\
		={}&		-\frac{\alpha_k\beta_{k+1}}{2}\nm{\lambda_{k+1}-\lambda^*}^2
		+\beta_{k}\dual{\lambda_{k+1}-\lambda_k,\lambda_{k+1}-\lambda^*}-\frac{\beta_{k}}{2}\nm{\lambda_{k+1}-\lambda_k}^2.
	\end{split}
	\]
	Plugging \eqref{eq:pdf-sys-im-lambda} into the second term and dropping the last negative term lead to
	\begin{equation}\label{eq:im-I2-est}
		I_2\leqslant-\frac{\alpha_k\beta_{k+1}}{2}\nm{\lambda_{k+1}-\lambda^*}^2+ \alpha_{k}\dual{Ax_{k+1}-b,\lambda_{k+1}-\lambda^*}+\dual{Ax_{k+1}-Ax_{k},\lambda_{k+1}-\lambda^*}.
	\end{equation}
	Similarly, we have
	\begin{equation}\label{eq:I3}
		\begin{split}
			I_3
			={}&
			\frac{\gamma_{k+1}-\gamma_k}{2}
			\nm{x_{k+1}-x^*}^2+	\frac{\gamma_{k}}{2}
			\left(\nm{x_{k+1}-x^*}^2-\nm{x_{k}-x^*}^2 
			\right)\\
			={}&		\frac{\alpha_k}{2}(\mu-\gamma_{k+1})
			\nm{x_{k+1}-x^*}^2+\gamma_{k}
			\dual{x_{k+1}-x_k, (x_{k+1}+x_{k})/2  -x^*}.
		\end{split}
	\end{equation}
	By \cref{eq:pk1}, we divide the last term by that
	\[
	\begin{aligned}
		\gamma_{k}
		\dual{x_{k+1}-x_k, (x_{k+1}+x_{k})/2  -x^*} 
		= {}&	\gamma_{k}
		\dual{x_{k+1}-x_k, x_{k+1} -x^*}
		-	\frac{\gamma_{k}}2
		\nm{x_{k+1}-x_k}^2\\
		=&-\alpha_k
		\dual{A^\top\lambda_{k+1}	-	p_{k+1}, x_{k+1} -x^*}
		-	\frac{\gamma_{k}}2\nm{x_{k+1}-x_k}^2.
	\end{aligned}
	\]
	Since \cref{eq:xk1} implies $	A^\top\lambda_{k+1}-	p_{k+1}\in\partial \mathcal L(x_{k+1},\lambda_{k+1})$, we obtain
	\[
	\begin{split}
		&-\alpha_k
		\dual{A^\top\lambda_{k+1}	-	p_{k+1}, x_{k+1} -x^*}\\
		\leqslant{}& \alpha_k(\mathcal L(x^*,\lambda_{k+1})-\mathcal L(x_{k+1},\lambda_{k+1}))
		-\frac{\mu\alpha_k}{2}\nm{x_{k+1}-x^*}^2\\
		={}& \alpha_k(\mathcal L(x^*,\lambda_{k+1})-\mathcal L(x_{k+1},\lambda^*))
		-\frac{\mu\alpha_k}{2}\nm{x_{k+1}-x^*}^2
		-\alpha_k\dual{Ax_{k+1}-b,\lambda_{k+1}-\lambda^*},
	\end{split}
	\]
	which promises the following bound 
	\begin{equation}\label{eq:I2im}
		\begin{split}
			I_3\leqslant {}&\alpha_k(\mathcal L(x^*,\lambda_{k+1})-\mathcal L(x_{k+1},\lambda^*))
			-\frac{\alpha_k\gamma_{k+1}}{2}\nm{x_{k+1}-x^*}^2\\
			{}&\quad-\alpha_k\dual{Ax_{k+1}-b,\lambda_{k+1}-\lambda^*}-	\frac{\gamma_{k}}2\nm{x_{k+1}-x_k}^2.
		\end{split}
	\end{equation}
	Consequently, combining \cref{eq:I1im,eq:im-I2-est,eq:I2im} proves \cref{eq:diff-Ek-im}.
	
	From \cref{eq:diff-Ek-im} we conclude that $	\mathcal E_k\leqslant  \mathcal E_0\times \prod_{i=0}^{k-1}\frac{1}{1+\alpha_i}$, which together with \cref{eq:Ek} implies \eqref{eq:conv-Lk-im} and that $\beta_0\nm{\lambda_k-\lambda^*}^2
	\leqslant 2\mathcal E_0$. Hence it is sufficient to prove 
	\eqref{eq:conv-Axk-b-im} and \eqref{eq:conv-fk-im}. From \cref{eq:lk1} follows that
	\begin{equation}\label{eq:lk-l0}
		\lambda_k-\frac1{\beta_k}(Ax_k-b) 
		= \lambda_0-\frac{1}{\beta_{0}}(Ax_0-b)
		\quad\text{for all } k\in\mathbb N.
	\end{equation}
	Then the estimate \eqref{eq:conv-Axk-b-im} is derived as below
	\[
	\begin{split}\nm{Ax_k-b}={}&\beta_k\nm{\lambda_k-\lambda_0
			+\beta_0^{-1}(Ax_0-b)}
		\leqslant{}
		\beta_k\nm{\lambda_k-\lambda_0}
		+\frac{\beta_k}{\beta_0}\nm{Ax_0-b}\\
		\leqslant {}&	\beta_k\nm{\lambda_k-\lambda^*}+\beta_k\nm{\lambda_0-\lambda^*}
		+\frac{\beta_k}{\beta_0}\nm{Ax_0-b}
		\leqslant \frac{\beta_k}{\beta_0}\mathcal R_0.
	\end{split}
	\]
	In addition, it is clear that
	\[
	0\leqslant 	\mathcal L(x_k,\lambda^*)-\mathcal L(x^*,\lambda_k)
	=f(x_k)-f(x^*)+\dual{\lambda^*,Ax_k-b}
	\leqslant
	\mathcal L(x_k,\lambda^*)-
	\mathcal L(x^*,\lambda_k),
	\]
	and thus
	\[
	\snm{f(x_k)-f(x^*)}\leqslant \nm{\lambda^*}\nm{Ax_k-b}+
	\mathcal L(x_k,\lambda^*)-
	\mathcal L(x^*,\lambda_k)
	\leqslant \frac{\beta_k}{\beta_0}\left(
	\mathcal{E}_0+\nm{\lambda^*}\mathcal R_0
	\right).
	\]
	This establishes \eqref{eq:conv-fk-im} and completes the proof of this theorem.
\end{proof}
\section{Composite Optimization}
\label{sec:comp}
In this section, we move to the composite case 
\begin{equation}\label{eq:min-h-g-Ax-b}
	\mathop{\min}_{x\in\R^n} f(x) = h(x)+g(x)\quad {\rm s.t.~} Ax = b,
\end{equation}
where $h$ is $L$-smooth and $\mu$-convex with $\mu\geqslant 0$ and $g$ is properly closed convex (possibly nonsmooth). Instead of the fully implicit scheme \cref{eq:pdf-sys-im-lambda}, to utilize the composite structure of $f = g+h$, we adopt a semi-implicit discretization that corresponds to the operator splitting (also known as the forward-backward technique). Note also that if $h$ is only convex but the nonsmooth part $g$ is $\mu$-convex, then we can always consider $f = \widehat{h} + \widehat{g}$ with $\widehat{h}(x) = h(x)+\mu/2\nm{x}^2$ and $\widehat{g}(x) = g(x) -\mu/2\nm{x}^2$, which agrees with the current assumption for \cref{eq:min-h-g-Ax-b} and $\prox_{\widehat{g}}$ can be computed by $\prox_{ g}$ (cf. \cite[Section 2.2]{parikh_proximal_2014}).
\subsection{A semi-implicit primal-dual proximal gradient method}
\label{sec:semi-comp}
Based on \cref{eq:pdf-sys-im-lambda}, we replace $\partial_x\mathcal L(x_{k+1},\lambda_{k+1})$ with $\nabla h(x_k)+\partial g(x_{k+1})+A^{\top}\lambda_{k+1}$ to obtain
\begin{subnumcases}{}
	{}v_{k+1} = x_{k}+\frac{x_{k+1}-x_{k}}{\alpha_k},
	\label{eq:pdf-sys-semi-v}\\
	\beta_{k+1} \frac{\lambda_{k+1}-\lambda_k}{\alpha_k} = {} \nabla_\lambda \mathcal L(	v_{k+1},\lambda_{k+1}),		
	\label{eq:pdf-sys-semi-lambda-comp}\\
	\gamma_{k+1} \frac{x_{k+1}-x_k}{\alpha_k} \in 
	-\nabla h(x_k)-\partial g(x_{k+1})-A^{\top}\lambda_{k+1},
	\label{eq:pdf-sys-semi-x-comp}
\end{subnumcases}
where the parameter system \cref{eq:beta} is discretized explicitly by 
\begin{equation}\label{eq:betak-ex}
	\frac{\gamma_{k+1}-\gamma_{k}}{\alpha_k}={}\mu-\gamma_{k},\quad
	\frac{\beta_{k+1}-\beta_{k}}{\alpha_k}=-\beta_{k}.
\end{equation}
Similar as before, we can rewrite \cref{eq:pdf-sys-semi-lambda-comp}
as a primal-dual formulation:
\begin{subnumcases}{}
	x_{k+1}
	={}\mathop{\argmin}_{x\in\R^n}
	\left\{f(x)+\dual{\nabla h(x_k)+A^\top\lambda_k,x}+
	\frac{1}{2\beta_{k+1}}\nm{Ax-b}^2+
	\frac{1}{2\eta_k}\nm{x-\widehat x_k}^2
	\right\},
	\label{eq:pdf-sys-semi-xk1-argmin}\\
	v_{k+1} = {}x_{k}+(x_{k+1}-x_k)/\alpha_k,\\
	\lambda_{k+1}={}\lambda_k + \alpha_k/\beta_{k+1}(Av_{k+1}-b),
	\label{eq:pdf-sys-semi-lk1-argmin}
\end{subnumcases}
where $\eta_k = \alpha_k/\gamma_{k+1}$ and $\widehat x_k =x_k
+\eta_k/\beta_{k}A^{\top}(Ax_k-b)$. In \eqref{eq:pdf-sys-semi-xk1-argmin}, the smooth part $h$ has been linearized while the nonsmooth part $g$ uses implicit discretization. This is similar with the proximal gradient method \cite{palomar_gradient-based_2009,parikh_proximal_2014}, and we have to impose proper restriction on the step size $\alpha_k$ (see Algorithm \ref{algo:Semi-PDPG}).

Notice also that the subproblem \eqref{eq:pdf-sys-semi-xk1-argmin} with respect to the primal variable $x_{k+1}$ is not easy to solve. From \eqref{eq:pdf-sys-semi-x-comp} we have $x_{k+1}	 = \prox_{\eta_k g}\left(x_k-\eta_k\nabla h(x_k)-\eta_k A^{\top}\lambda_{k+1}\right)$, and putting this into \eqref{eq:pdf-sys-semi-lk1-argmin} gives
\begin{equation}
	\label{eq:lk1-semi-comp-prox}
	\beta_{k+1}\lambda_{k+1} - A \prox_{\eta_k g}\left(y_k-\eta_k A^{\top}\lambda_{k+1}\right) =z_k,
\end{equation}
where $y_k={}x_k-\eta_k\nabla h(x_k)$ and $ z_k  ={} \beta_{k+1}\left(\lambda_k-\beta_k^{-1}(Ax_k-b)\right)-b$. Below, we present a semi-smooth Newton method to solve the nonlinear equation \cref{eq:lk1-semi-comp-prox} in terms of the multiplier $\lambda_{k+1}$. This can be very efficient for some practical cases that (i) the multiplier has lower dimension than the primal variable; (ii) the problem \cref{eq:lk1-semi-comp-prox} itself possesses some nice properties such as semi-smoothness and simple closed proximal formulation of $g$; (iii) efficient iterative methods for updating  the Newton direction can be considered if there has sparsity. 
\subsubsection{A semi-smooth Newton method for the subproblem \cref{eq:lk1-semi-comp-prox}}
\label{sec:SsN}
Define a mapping $F_k:\R^m\to\R^m$ by that
\begin{equation}\label{eq:Fk}
	F_k(\lambda) := \beta_{k+1}\lambda - A \prox_{\eta_k g}\left(y_k-\eta_k A^{\top}\lambda\right) -z_k\quad\forall\,\lambda\in\R^m.
\end{equation}
Then \eqref{eq:lk1-semi-comp-prox} is equivalent to $F_k(\lambda_{k+1}) = 0$. By Moreau's identity (cf. \cite[Theorem 6.45]{Scheinberg}) 
\begin{equation}\label{eq:moreau-id}
	\prox_{\eta g}(x)+\eta\prox_{g^*/\eta}(x/\eta) =x,
\end{equation}
where $g^*$ denotes the conjugate function of $g$, we find 
that $F_k(\lambda) = \nabla \mathcal F_k(\lambda)$, where 
\begin{equation}\label{eq:cal-Fk}\small
	\begin{split}
		\mathcal F_k(\lambda) := {}&
		\frac{\beta_{k+1}}{2}\nm{\lambda}^2-\dual{z_k,\lambda}
		+g^*\left(\prox_{g^*/\eta_k}(y_k/\eta_k-A^{\top}\lambda)\right)
		+\frac{1}{2\eta_k}\nm{\prox_{\eta_k g}(y_k-\eta_kA^{\top}\lambda)}^2.
	\end{split}
\end{equation}

Let $\partial \prox_{\eta_k g}(\lambda)$ be the generalized Clarke subdifferential \cite{clarke_optimization_1987} of $\prox_{\eta_k g}(\lambda)$. If $P_k(\lambda)\in\partial \prox_{\eta_k g}\left(y_k-\eta_k A^{\top}\lambda\right)$ is symmetric (this is indeed true when $g$ is either the indicator function or the support function for some nonempty convex polyhedral \cite{Han1997}), then for any $\lambda\in\R^m$ we can define an SPD matrix
\begin{equation}\label{eq:JFk-semi-comp}
	JF_k(\lambda) := \beta_{k+1}I+\eta_kAP_k(\lambda)A^{\top}\in\R^{m\times m}.
\end{equation}
The semi-smooth Newton (SsN) method for solving \eqref{eq:lk1-semi-comp-prox} reads as follows: given an initial guess $\lambda^0\in\R^m$, do the iteration 
\begin{equation}\label{eq:SsN}
	\lambda^{j+1} = \lambda^j-\left[JF_k(\lambda^j)\right]^{-1}F_k(\lambda^j),\quad j\geq 0.
\end{equation}
Theoretically, it possesses local superlinear convergence provided that $F_k$ is semismooth \cite{Qi1993,Qi1993a}. Practically, it can be terminated under some suitable criterion and for global convergence,  a line search procedure \cite{dennis_numerical_1996} shall be supplemented: given a Newton direction $d^j = -\left[JF_k(\lambda^j)\right]^{-1}F_k(\lambda^j)$ at step $j$, find the smallest nonnegative integer $r\in\mathbb N$ such that
\begin{equation}\label{eq:line-search}
	\mathcal F_k(\lambda^j+\delta^rd^j)\leqslant \mathcal F_k(\lambda^j)+\nu\delta^r\dual{F_k(\lambda^j),d^j},
\end{equation}
where $\nu\in(0,1/2),\,\delta\in(0,1]$ and $\mathcal F_k$ has been defined in \cref{eq:cal-Fk}. Generally the inverse 
operation $\left[JF_k(\lambda^j)\right]^{-1}$ in \cref{eq:SsN} shall be approximated by 
some iterative process such as the (preconditioned) conjugate 
gradient method \cite{saad_iterative_2003}. For more discussions about the linear solver for $d^j$, we refer to Section \ref{sec:PCG-l1l2}.

Below we summarize the semi-implicit scheme \cref{eq:pdf-sys-semi-x-comp} as an algorithm framework, which is called the {\it semi-implicit primal-dual proximal gradient} (Semi-PDPG) method. As suggested later by \cref{thm:conv-semi-comp}, the step size $\alpha_k$ is determined simply by $\alpha_k(L+\gamma_{k+1}) = \gamma_{k+1}$, which promises the convergence rate $\mathcal O(\min\{L/k,(1+\mu/L)^{-k}\})$ (cf. \cref{eq:conv-comp}).
\begin{algorithm}[H]
	\caption{Semi-PDPG method for \cref{eq:min-h-g-Ax-b} with $h$ being $L$-smooth and $\mu$-convex $(\mu\geqslant 0)$}
	\label{algo:Semi-PDPG}
	\begin{algorithmic}[1] 
		\REQUIRE  $\gamma_0>0,\,\beta_0>0,\, x_0
		\in\R^n,\,\lambda_0\in\R^m$.
		\FOR{$k=0,1,\ldots$}
		\STATE Set $\sigma_k=L+2\gamma_k-\mu$ and $\Delta_k=\sigma_k+\sqrt{\sigma_k^2+4\gamma_k(\mu-\gamma_k)}$.
		\STATE Compute the step size $\alpha_k = 2\gamma_{k}/\Delta_k\in(0,1)$.
		\STATE Update $\displaystyle \beta_{k+1}= \beta_k(1-\alpha_k)$ and $\displaystyle \gamma_{k+1} =\mu\alpha_k+(1-\alpha_k)\gamma_k$.
		\STATE Set $\displaystyle 		\eta_{k}={}\alpha_k/\gamma_{k+1}$ and $y_k=x_k-\eta_k\nabla h(x_k)$.
		\STATE  Set $\displaystyle z_k  ={} \beta_{k+1}\left(\lambda_k-\beta_k^{-1}(Ax_k-b)\right)-b$.
		\STATE Solve $\lambda_{k+1}$ from \eqref{eq:lk1-semi-comp-prox}
		via the SsN iteration \eqref{eq:SsN} with the line search procedure \eqref{eq:line-search}.
		\STATE Update $x_{k+1}	 = \prox_{\eta_k g}\left(y_k-\eta_k A^{\top}\lambda_{k+1}\right)$.
		\ENDFOR
	\end{algorithmic}
\end{algorithm}
\subsection{Proof of the convergence rate }
To move on, the following  two lemmas are needed.
\begin{lem}\label{lem:key-gd-map}
	Assume	$h$ is $L$-smooth and $\mu$-convex with $\mu\geqslant 0$ and $g$ is properly closed convex. 
	Let $\{(x_k,\lambda_k,\gamma_k,\beta_k)\}$ be 
	generated by \cref{eq:pdf-sys-semi-lambda-comp,eq:betak-ex}, then for all $y\in\R^n$,
	\begin{equation}\label{eq:coer-est}
		\begin{split}
			{}&\mathcal L(x_{k+1},\lambda_{k+1})-\mathcal L(y,\lambda_{k+1})	
			+\frac{\gamma_{k+1}}{\alpha_k}\dual{x_{k+1}-x_k,x_{k}-y}\\
			\leqslant{}&	-\frac{\mu}{2}\nm{y-x_{k}}^2
			+\frac{L\alpha_k-2\gamma_{k+1}}{2\alpha_k}\nm{x_{k+1}-x_k}^2.
		\end{split}
	\end{equation}
\end{lem}
\begin{proof}
	Define $	\phi(x):=	h(x) + \dual{\lambda_{k+1},Ax-b}$ for all $x\in\R^n$.
	As $h$ is $L$-smooth and $\mu$-convex, there holds that
	\[
	\begin{split}
		\phi(x_k)-\phi(y)+\dual{\nabla \phi(x_k),y-x_k}
		\leqslant&-\frac{\mu}{2}\nm{y-x_k}^2,\\
		\phi(x_{k+1})-	\phi(x_k)-\dual{\nabla \phi(x_k),x_{k+1}-x_k}
		\leqslant{}& \frac{L}{2}\nm{x_{k+1}-x_k}^2.
	\end{split}
	\]
	In addition, by \eqref{eq:pdf-sys-semi-x-comp}, we have
	\[
	\gamma_{k+1} \frac{x_k-x_{k+1}}{\alpha_k} -\nabla \phi(x_k)\in\partial g(x_{k+1}),
	\]
	and it follows that
	\[
	\begin{split}
		{}&	g(x_{k+1})-g(y)\leqslant \dual{		\gamma_{k+1} \frac{x_k-x_{k+1}}{\alpha_k} 
			-\nabla \phi(x_k),x_{k+1}-y}\\
		={}&\frac{\gamma_{k+1}}{\alpha_k}\dual{x_k-x_{k+1},x_{k}-y} -
		\dual{	\nabla \phi(x_k),x_{k+1}-y}-\frac{\gamma_{k+1}}{\alpha_k}\nm{x_{k+1}-x_k}^2.
	\end{split}
	\]
	Collecting the above estimates and using the fact $\mathcal L(\cdot,\lambda) = \phi(\cdot)+g(\cdot)$, we obtain \cref{eq:coer-est} and conclude the proof.
\end{proof}
Recall $\gamma_{\min}$ defined in \cref{eq:bd-gama} and for later use we set $	\gamma_{\max}: = \max\{\gamma_0,\mu\}$.
\begin{lem}
	\label{lem:rate-bk}	
	Let $\{(\gamma_k,\beta_k)\}$ be defined by \cref{eq:betak-ex} with $\alpha_k(L+\gamma_{k+1}) \leqslant 2\gamma_{k+1}$, then $\alpha_k\in(0,1]$ for all $k\in\mathbb N$. Moreover, if $\alpha_k(L+\gamma_{k+1}) =\gamma_{k+1}$ then
	\begin{equation}\label{eq:rate-bk}
		\prod_{i=0}^{k-1}(1-\alpha_i)\leqslant \min\left\{
		\frac{L+\gamma_{\max}}{\gamma_0 k+L+\gamma_{\max}}
		,\,
		\left(
		\frac{L}{L+\gamma_{\min}}
		\right)^k
		\right\}.
	\end{equation}
\end{lem}
\begin{proof}
	Let us first verify the existence of the sequence $\{\alpha_k\} \subset(0,1]$. 
	As $\alpha_k(L+\gamma_{k+1}) \leqslant 2\gamma_{k+1}$ and $	\gamma_{k+1} =\gamma_{k}+\alpha_k(\mu-\gamma_k)$ (cf. \cref{eq:betak-ex}), we obtain $\psi_k(\alpha_k):=	(\mu-\gamma_k)\alpha_k^2
	+(L+3\gamma_k-2\mu)\alpha_k-2\gamma_{k}\leqslant 0$. As $\gamma_0>0$, we have $\psi_0(0)=-2\gamma_{0}<0$ and $\psi_0(1)=L-\mu\geqslant  0$. Hence, there must be at least one (actually unique) root $\alpha^*\in(0,1]$ 
	of $\psi_0(\alpha)=0$. Hence, any $\alpha_0\in(0,\alpha^*]$ satisfies $\alpha_0(L+\gamma_{1})\leqslant 2\gamma_{1}$.
	Repeating this process for $\psi_k(\alpha)$ and noticing that $\gamma_k>0$ yield the existence 
	of $\alpha_k\in(0,1]$ for all $k\geqslant 1$.
	
	From \cref{eq:betak-ex} we have $\beta_{k} = \beta_0 		\prod_{i=0}^{k-1}(1-\alpha_i)$. 
	It remains to investigate the asymptotic decay behavior of $\beta_{k}$ with $\alpha_k(L+\gamma_{k+1}) =\gamma_{k+1}$. 
	Let us start from the identity
	\[
	\frac{1}{\beta_{k+1}}-\frac{1}{\beta_{k}}
	=\frac{\beta_{k}-\beta_{k+1}}{\beta_{k} \beta_{k+1}}
	=\frac{\alpha_{k}}{\beta_{k+1}}.
	\]
	Besides, we have
	\[
	\frac{\gamma_{k+1}}{\gamma_{k}}\geqslant 1-\alpha_k = 
	\frac{\beta_{k+1}}{\beta_{k}}\quad \Longrightarrow\quad
	\gamma_{k}\geqslant \frac{\gamma_0}{\beta_0}\beta_{k}.
	\]
	It follows from this and the relation $	\alpha_k(L+\gamma_{k+1}) = \gamma_{k+1}$ that 
	\[
	\frac{1}{\beta_{k+1}}-\frac{1}{\beta_{k}}\geqslant 
	\frac{\gamma_0\alpha_{k}}{\beta_0\gamma_{k+1}} = 
	\frac{\gamma_0}{\beta_0 (L+\gamma_{k+1})}
	\geqslant 	\frac{\gamma_0}{\beta_0 (L+\gamma_{\max})}.
	\]
	Hence, we get 
	\begin{equation}\label{eq:est-betak-1}
		\frac{	\beta_{k}}{\beta_{0}}\leqslant 
		\frac{L+\gamma_{\max}}{\gamma_0 k+L+\gamma_{\max}}.
	\end{equation}
	On the other hand, since $\gamma_{k+1}\geqslant \gamma_{\min}$, we have $\alpha_k= \gamma_{k+1}/(L+\gamma_{k+1})\geqslant
	\gamma_{\min}/(L+\gamma_{\min})$.
	Therefore, another bound follows
	\[
	\frac{	\beta_{k}}{\beta_{0}} = \prod_{i=0}^{k-1}(1-\alpha_i)\leqslant 
	\left(
	\frac{L}{L+\gamma_{\min}}
	\right)^k.
	\]
	Combining this with \cref{eq:est-betak-1} establishes \cref{eq:rate-bk} and completes the proof of this lemma.
\end{proof}
We now prove the convergence rate of Algorithm \ref{algo:Semi-PDPG} 
by using the Lyapunov function \eqref{eq:Ek}. 
\begin{thm}\label{thm:conv-semi-comp}
	Assume	$h$ is $L$-smooth and $\mu$-convex with $\mu\geqslant 0$ and $g$ is properly closed convex. Let $\{(x_k,\lambda_k,\gamma_k,\beta_k)\}$ be 
	generated by \cref{eq:pdf-sys-semi-lambda-comp,eq:betak-ex} with $\alpha_k(L+\gamma_{k+1}) \leqslant 2\gamma_{k+1}$, then we have $\{\alpha_k\}\subset(0,1]$ and 
	\begin{equation}\label{eq:conv-semi-comp}
		\mathcal E_{k+1}-\mathcal E_k
		\leqslant-\alpha_k\mathcal E_{k},
		\quad\text{for all }\,k\in\mathbb N.
	\end{equation}
	Moreover, if $\alpha_k(L+\gamma_{k+1}) =\gamma_{k+1}$, then it holds that
	\begin{equation}\label{eq:conv-comp}\small
		{}\mathcal L(x_{k},\lambda^*)-	\mathcal L(x^*,\lambda_{k})+\snm{F(x_k)-F(x^*)}+\nm{Ax_{k}-b}\leqslant C_0\times 	\min\left\{
		\frac{L+\gamma_{\max}}{\gamma_0 k+L+\gamma_{\max}}
		,\,
		\left(
		\frac{L}{L+\gamma_{\min}}
		\right)^k
		\right\},
	\end{equation}
	where 
	$C_0:= \mathcal E_0+\mathcal R_0(1+\nm{\lambda^*})$ with $\mathcal R_0:=
	\sqrt{2\beta_{0}\mathcal E_0}+\beta_{0}
	\nm{\lambda_0-\lambda^*}+\nm{Ax_0-b}$.
\end{thm}
\begin{proof}
	The existence of the step size sequence $\{\alpha_k\}\subset(0,1]$ has been proved in \cref{lem:rate-bk}.
	Once the contraction \eqref{eq:conv-semi-comp} is established, we obtain $\mathcal E_k\leqslant \mathcal E_0\times \prod_{i=0}^{k-1}(1-\alpha_i)$, and the estimate \cref{eq:conv-comp}
	can be 
	obtain by using \cref{lem:rate-bk} and the same procedure for \eqref{eq:conv-Axk-b-im}, \eqref{eq:conv-Lk-im} and \eqref{eq:conv-fk-im}. 
	
	Following the proof of \cref{thm:conv-im}, we start from the 
	difference $	\mathcal E_{k+1}-\mathcal E_k = I_1+I_2+I_3$, where $I_1,\,I_2$ and $I_3$ are defined in \cref{eq:Ii}. By \cref{eq:betak-ex}, we have
	\[
	\begin{split}
		I_2
		={}&		\frac{\beta_{k+1}-\beta_{k}}{2}\nm{\lambda_{k}-\lambda^*}^2
		+\frac{\beta_{k+1}}{2}
		\left(\nm{\lambda_{k+1}-\lambda^*}^2 - 
		\nm{\lambda_{k}-\lambda^*}^2\right)\\
		={}&		-\frac{\alpha_k\beta_{k}}{2}\nm{\lambda_{k}-\lambda^*}^2
		+\beta_{k+1}\dual{\lambda_{k+1}-\lambda_k,\lambda_{k+1}-\lambda^*}-\frac{\beta_{k+1}}{2}\nm{\lambda_{k+1}-\lambda_k}^2.
	\end{split}
	\]
	Plugging \eqref{eq:pdf-sys-semi-lambda-comp} into the second term and dropping the last negative term lead to
	\begin{equation}\label{eq:im-I2-est-comp}
		I_2\leqslant-\frac{\alpha_k\beta_{k}}{2}\nm{\lambda_{k}-\lambda^*}^2+ \alpha_{k}\dual{Av_{k+1}-b,\lambda_{k+1}-\lambda^*}.
	\end{equation}
	Similarly, for $I_3$, it holds that
	\[
	\begin{split}
		I_3
		={}&		\frac{\gamma_{k+1}-\gamma_k}{2}
		\nm{x_{k}-x^*}^2+\frac{\gamma_{k+1}}{2}
		\left(\nm{x_{k+1}-x^*}^2-\nm{x_{k}-x^*}^2 
		\right)\\
		=	{}&	\frac{\alpha_k(\mu-\gamma_{k})}{2}
		\nm{x_{k}-x^*}^2+	\gamma_{k+1}
		\dual{x_{k+1}-x_k, x_{k} -x^*}
		+	\frac{\gamma_{k+1}}2
		\nm{x_{k+1}-x_k}^2,
	\end{split}
	\]
	and invoking Lemma \ref{lem:rate-bk} gives
	\[
	\begin{split}
		I_3 \leqslant  {}&\alpha_k(\mathcal L(x^*,	\lambda_{k+1} )-\mathcal L(x_{k+1},	\lambda_{k+1} ))
		-\frac{\alpha_k \gamma_{k}}{2}\nm{x_{k}-x^*}^2
		+\frac{L\alpha_k-\gamma_{k+1}}{2}\nm{x_{k+1}-x_k}^2.
	\end{split}
	\]
	
	To match the right hand side of \eqref{eq:conv-semi-comp}, we shift $(x_{k+1},\lambda_{k+1})$ 
	to $(x_{k},\lambda_{k+1})$ and then to $(x_{k},\lambda^*)$ and obtain that
	\[
	\begin{aligned}
		I_3\leqslant {}&\alpha_k(\mathcal L(x^*,\lambda_k)
		-\mathcal L(x_{k},\lambda^*))
		-\frac{\alpha_k\gamma_{k}}{2}\nm{x_k-x^*}^2\\
		{}&\quad-\alpha_k\dual{Ax_{k}-b,	\lambda_{k+1} -\lambda^*}
		+\frac{L\alpha_k-\gamma_{k+1}}{2}\nm{x_{k+1}-x_k}^2\\
		{}&\quad \qquad+\alpha_k(\mathcal L(x_{k},	\lambda_{k+1} )
		-\mathcal L(x_{k+1},	\lambda_{k+1} )).
	\end{aligned}
	\]
	To offset the last term in the above estimate, we shall divide $I_1$ as follows
	\[
	\begin{split}
		{}&I_1 = 	\mathcal L(x_{k+1},\lambda^*)-\mathcal L(x_{k},\lambda^*)\\
		={}&\alpha_k(\mathcal L(x_{k+1},\lambda_{k+1})-\mathcal L(x_{k},\lambda_{k+1}))
		-	 \dual{Ax_{k+1}-Ax_k,\lambda_{k+1}-\lambda^*}\\
		{}&\quad+(1-\alpha_k)(\mathcal L(x_{k+1},\lambda_{k+1})-\mathcal L(x_{k},\lambda_{k+1})).
	\end{split}
	\]
	Applying Lemma \ref{lem:rate-bk} again implies
	\[
	\begin{split}
		I_1 \leqslant{}&\alpha_k(\mathcal L(x_{k+1},\lambda_{k+1})-\mathcal L(x_{k},\lambda_{k+1}))
		-	 \dual{Ax_{k+1}-Ax_k,\lambda_{k+1}-\lambda^*}\\
		{}&\qquad+\frac{1-\alpha_k}{2\alpha_k}(L\alpha_k-2\gamma_{k+1} )\nm{x_{k+1}-x_k}^2,
	\end{split}
	\]
	which together with the relation $v_{k+1} = x_k+(x_{k+1}-x_k)/\alpha_k$ yields that
	\begin{equation}\label{eq:I1+I3-semi-comp}
		\begin{split}
			I_1 +I_3\leqslant{}&\alpha_k(\mathcal L(x^*,\lambda_k)
			-\mathcal L(x_{k},\lambda^*))
			-\frac{\alpha_k\gamma_{k}}{2}\nm{x_k-x^*}^2
			- \alpha_{k}\dual{Av_{k+1}-b,\lambda_{k+1}-\lambda^*}\\
			{}&\qquad
			+\frac{\alpha_k(L+\gamma_{k+1})-2\gamma_{k+1}}{2\alpha_k}
			\nm{x_{k+1}-x_k}^2.
		\end{split}
	\end{equation}
	Consequently, combining this with the estimate \cref{eq:im-I2-est-comp} for $I_2$ implies
	\[
	\begin{split}
		\mathcal E_{k+1}-\mathcal E_k
		\leqslant&-\alpha_k\mathcal E_k
		+\frac{\alpha_k(L+\gamma_{k+1})-2\gamma_{k+1}}{2\alpha_k}
		\nm{x_{k+1}-x_k}^2.
	\end{split}
	\]
	As $\alpha_k(L+\gamma_{k+1}) \leqslant 2\gamma_{k+1}$, this establishes \eqref{eq:conv-semi-comp} and completes the proof.
\end{proof}
\section{Numerical Experiments}
\label{sec:numer}
In this part, we investigate practical performances of Algorithms \ref{algo:Im-PD} and \ref{algo:Semi-PDPG} for the $l_1$-$l_2$ minimization \cref{eq:l1-l2} and the total-variation based image denoising model \cref{eq:rof-dis-x-y-z}. 
\subsection{The $l_1$-$l_2$ minimization}
\label{sec:l1l2}
We first consider the linearly constrained $l_1$-$l_2$ minimization:
\begin{equation}\label{eq:l1-l2}
	\min_{x\in\R^n}~\frac{\rho}{2}\nm{x}^2+\nm{x}_1\quad {\rm s.t.~} Ax = b,
\end{equation}
where $\rho>0,\,b\in\R^m$ and $A\in\R^{m\times n}$ with $m\ll n$. This is a regularized model for the so-called basis pursuit \cite{chen1999}, which corresponds to the limit case $\rho=0$ and is related to compressed sensing \cite{Candes2008}.

Let $g(x) = \nm{x}_{1}$, then for any $\eta>0$, the proximal mapping $\prox_{\eta g}(x)={\rm sgn}(x)\odot\max\{|x|-\eta,0\}$
is well known as the {\it soft thresholding operator}, with the $i$-th component of  $y=\prox_{\eta g}(x)$ being given by $y_i = {\rm sgn}(x_i)\max\{|x_i|-\eta,0\}$. Here and in what follows, $\odot$ and $\oslash$ stand respectively for element-wise multiplication and division operations. The conjugate function $g^*$ of $g$ is the indicator function of the cube $[-1,1]^n$ and thus $\prox_{\eta g^*}(x) = \min\left\{
\max\{x,-1\},1
\right\}$.
\subsubsection{Comparison with ALB}
There are some well-known Bregman methods for solving \cref{eq:l1-l2}; see \cite{yin_analysis_2010,Huang2013,Kang2013,cai_linearized_2009}. Both of the two accelerated variants in \cite{Huang2013,Kang2013} possess the nonergodic sublinear  rate $\mathcal O(1/k^2)$ for the dual objective but the method in \cite{Kang2013} involves a subproblem for the primal variable. In contrast, the accelerated linearized Bregman (ALB) method in \cite{Huang2013} linearizes the augmented term and admits closed update formulation in each step. More precisely, it reads as follows: given $\lambda_0,\,\widetilde\lambda_0\in\R^{m}$, do the iteration 
\begin{equation}\label{eq:ALB}
	\left\{
	\begin{split}
		x_{k+1} = {}&\prox_{ g/\rho}\left(-A^{\top}\widetilde{\lambda}_k/\rho\right),\\
		\lambda_{k+1} = {}&\widetilde \lambda_k + \tau\left(Ax_{k+1}-b\right),\\
		\widetilde\lambda_{k+1} = {}&t_k\lambda_{k+1}+(1-t_k)\lambda_k,
	\end{split}
	\right.
\end{equation}
where $t_k = (2k+3)/(k+3)$ and $\tau= \rho/\nm{A}^2$. 
\renewcommand\arraystretch{1.2}
\begin{table}[H]
	\centering
	\caption{Performances of Inexact Semi-PDPG (i.e.,Algorithm \ref{algo:Semi-PDPG-l1l2}) and 
		ALB method \cref{eq:ALB} for solving \eqref{eq:l1-l2}. Here, ``direct" and ``PCG" mean that the linear system in step \ref{algo:d-l1l2} of Algorithm \ref{algo:Semi-PDPG-l1l2} is solved respectively by direct method and PCG.}
	\label{tab:l1-l2}
	\small\setlength{\tabcolsep}{0.8pt}
	\begin{tabular}{cccccccccccccccc}
		\toprule
		&\phantom{a}&&\phantom{a}&\phantom{a} & \phantom{a} &
		\multicolumn{3}{c}{Inexact Semi-PDPG(direct)}&\phantom{a} & 
		\multicolumn{3}{c}{Inexact Semi-PDPG(PCG)}&\phantom{a} & 
		\multicolumn{2}{c}{ALB}\\
		\cmidrule{7-9}			\cmidrule{11-13} 	\cmidrule{15-16} 
		&&$m$& &$n$ &&its& SsN&time(sec)&&its& SsN&time(sec)&    & its   & time(sec)\\
		\midrule
		\multirow{3}{*}{ $\rho = 0.5$}	
		&& 5e+02  && 2e+03  && 21 &42& 5.20&& 21 &40& 3.59&& 537   & 4.20 \\ 
		&& 8e+02  && 3e+03  && 21 &46& 10.76&& 21 &43& 11.40&& 593   & 10.86 \\ 
		&& 1e+03  && 4e+03  && 21 &39& 12.33&& 21 &42& 12.37&& 546   & 16.15 \\ 
		\midrule
		\multirow{3}{*}{$\rho =0.1$}						
		&& 2e+02  && 1e+03  && 20 &34& 0.70&& 20 &43& 1.08&& 2330   & 2.68 \\ 
		&& 5e+02  && 3e+03  && 21 &37& 3.66&& 19 &51& 5.66&& 1967   & 23.81 \\ 
		&& 1e+03  && 5e+03  && 20 &43& 15.61&& 20 &47& 14.00&& 2118   & 81.83 \\ 
		\midrule
		\multirow{3}{*}{ $\rho = 0.01$}						
		&& 5e+02  && 2e+03  && 19 &56& 4.50&& 18 &60& 5.92&& 13174   & 103.54 \\ 
		&& 9e+02  && 4e+03  && 18 &56& 12.49&& 22 &87& 41.76&& 12712   & 379.83 \\ 
		&& 2e+03  && 8e+03  && 17 &63& 87.29&& 19 &82& 246.34&& 13819   & 1693.99 \\ 
		\midrule
		\multirow{3}{*}{ $\rho =0.005$}						
		&& 8e+02  && 3e+03  && 21 &86& 23.39&& 19 &75& 39.25&& 19793   & 375.07 \\ 
		&& 2e+03  && 6e+03  && 20 &86& 153.48&& 23 &126& 579.69&& 20811   & 1778.27 \\ 
		&& 3e+03  && 9e+03  && 19 &83& 509.93&& 24 &139& 1933.28&& 21568   & 6592.60 \\ 
		\bottomrule
	\end{tabular}
\end{table}
We apply Algorithm \ref{algo:Semi-PDPG} to the problem \cref{eq:l1-l2}. In this case, as $g$ is piecewise affine, $\prox_{\eta g}$ is strongly semismooth \cite{Facchinei2006} and so is the nonlinear mapping $F_k(\cdot)$ defined by \cref{eq:Fk}.
For $\eta>0$ and $x\in\R^n$, define a diagonal matrix 
\begin{equation}\label{eq:P-eta}
	P_\eta(x)= {\rm diag}(p)\in\R^{n\times n}\quad\text{with } p_i = \left\{
	\begin{aligned}
		{}&1&&\text{if }\snm{x_i}\geqslant \eta,\\
		{}&0&&\text{if }\snm{x_i}<\eta.
	\end{aligned}
	\right.
\end{equation}
Then it is easy to see 
that $P_\eta(x)\in\partial \prox_{\eta g}(x)$, 
and we obtain a generalized Clarke subgradient for \cref{eq:lk1-semi-comp-prox}:
\begin{equation}\label{eq:JFk}
	JF_k(\lambda) = \beta_{k+1}I+\eta_k AP_{\eta_k}[v_k(\lambda)]A^{\top}\in\R^{m\times m},
\end{equation}
where $v_k(\lambda) = y_k-\eta_kA^{\top}\lambda$. Note that $P_{\eta_k}[v_k(\lambda)] = {\rm diag}(p)$ where $p$ is defined by \cref{eq:P-eta} with $p_i\in\{0,1\}$, and thus $JF_k(\lambda)$ is always SPD. Moreover, the function \cref{eq:cal-Fk} becomes
\[
\mathcal F_k(\lambda) = {}
\frac{\beta_{k+1}}{2}\nm{\lambda}^2-\dual{z_k,\lambda}+\frac{1}{2\eta_k}\nm{\prox_{\eta_k g}[v_k(\lambda)]}^2.
\]

\begin{algorithm}[H]
	\caption{Inexact Semi-PDPG method for the $l_1$-$l_2$ minimization problem \cref{eq:l1-l2}}
	\label{algo:Semi-PDPG-l1l2}
	\begin{algorithmic}[1] 
		\REQUIRE  $\gamma_0>0,\,\beta_0>0,\,\,x_0
		\in\R^{n}$ and $\lambda_0\in\R^{m}$.
		\STATE Problem setting: $\rho>0,\,b\in\R^m$ and $A\in\R^{m\times n}$.
		\STATE SsN setting: $\nu = 0.2,\,\delta = 0.9$ and $j_{\max}= 10$.
		\STATE Tolerances: $\mathtt{KKT\_Tol} = 10^{-6}$ and $\mathtt{SsN\_Tol} =10^{-8}$.
		\FOR{$k=0,1,\ldots$}
		\STATE Set $\sigma_k=2\gamma_k$ and $\Delta_k=\sigma_k+\sqrt{\sigma_k^2+4\gamma_k(\rho-\gamma_k)}$.
		\STATE Compute the step size $\alpha_k = 2\gamma_{k}/\Delta_k\in(0,1)$.
		\STATE Update $\displaystyle \beta_{k+1}= \beta_k(1-\alpha_k)$ and $\displaystyle \gamma_{k+1} =\rho\alpha_k+(1-\alpha_k)\gamma_k$.
		\STATE Set $\displaystyle 		\eta_{k}={}\alpha_k/\gamma_{k+1}$ and $y_k=x_k-\eta_k\rho x_k$.
		\STATE  Set $\displaystyle z_k  ={} \beta_{k+1}\left(\lambda_k-\beta_k^{-1}(Ax_k-b)\right)-b$.
		\STATE Solve $\lambda_{k+1}$ from the nonlinear equation
		\begin{equation}\label{eq:iner-l1l2}
			F_k(\lambda):=	\beta_{k+1}\lambda - A \prox_{\eta_k g}\left(y_k-\eta_k A^{\top}\lambda \right) -z_k=0
		\end{equation}
		via the following SsN iteration with $\lambda = \lambda_k$ and $j = 0$:
		\WHILE[SsN iteration]{$\nm{F_k(\lambda)}>\mathtt{SsN\_Tol}$ and $j<j_{\max}$}\label{algo:ssn-first}
		\STATE Compute $v_k = y_k - \eta_k A^\top\lambda$.
		\STATE Find $P_{\eta_k}(v_k)\in\partial \prox_{\eta_kg}(v_k)$ via \cref{eq:P-eta}.
		\STATE Compute $JF_k(\lambda)= \beta_{k+1}I+\eta_k AP_{\eta_k}(v_k)A^{\top}$.
		\STATE Solve $JF_k(\lambda) d =- F_k(\lambda)$.\label{algo:d-l1l2}
		\STATE Find the smallest integer $r\in\mathbb N_+$ such that $		\mathcal F_k(\lambda+\delta^rd)\leqslant \mathcal F_k(\lambda)+\nu\delta^r\dual{F_k(\lambda), d}$.
		\STATE Update $\lambda = \lambda+\delta^rd$ and $j = j+ 1$.
		\ENDWHILE		\label{algo:ssn-end}
		\STATE Update $\lambda_{k+1} = \lambda$ and $x_{k+1}	 = \prox_{\eta_k g}\left(y_k-\eta_k A^{\top}\lambda_{k+1}\right)$.
		\IF{${\rm Res}(k)\leqslant \mathtt{KKT\_Tol}$}
		\STATE {\bf break}
		\ENDIF
		\ENDFOR
	\end{algorithmic}
\end{algorithm}

We rewrite Algorithm \ref{algo:Semi-PDPG} in Algorithm \ref{algo:Semi-PDPG-l1l2}, where a practical inexact setting is considered. 
The SsN iteration (see lines \ref{algo:ssn-first}--\ref{algo:ssn-end} in Algorithm \ref{algo:Semi-PDPG-l1l2}) is stopped either $\nm{F_k(\lambda)}\leqslant \mathtt{SsN\_Tol} = 10^{-8}$ or $j_{\max}=10$. For the line search procedure, we adopt $\nu = 0.2$ and $\delta = 0.9$. All initial guesses $\beta_0,\,x_0$ and $\lambda_0$ are generated randomly, and we chose $\gamma_{0} = \mu + \sigma$ with $\sigma$ obeying the uniform distribution on $[0,1]$. By \cref{thm:conv-semi-comp} we the linear rate $2^{-k}$ (with exact computation). 

Recall the optimality condition of problem \cref{eq:l1-l2}: $Ax^* = b$ and $x^* = \prox_{ g}((1-\rho)x^*-A^\top\lambda^*)$.
Hence, we consider the stopping criterion:
\begin{equation}\label{eq:res-kkt}
	{\rm Res}(k) := \max\left\{{\rm Res}(x_k),{\rm Res}(\lambda_k)\right\}\leqslant \mathtt{KKT\_Tol} = 10^{-6},
\end{equation}
where the relative KKT residuals are defined by 
\[
{\rm Res}(\lambda_k):= \frac{\nm{Ax_k-b}}{1+\nm{b}}\quad\text{and}\quad
{\rm Res}(x_k):=	\frac{\nm{x_k-\prox_{g}\left((1-\rho)x_k-A^{\top}\lambda_k\right)}}{1+\nm{x_k}}.
\]

In step \ref{algo:d-l1l2} of Algorithm \ref{algo:Semi-PDPG-l1l2}, we have to solve a linear system and we consider two ways: one is direct method as $m\ll n$ and the other is preconditioned conjugate gradient (PCG) method (cf.\cite[Algorithm 9.1]{saad_iterative_2003}) with diagonal preconditioner. The PCG iteration is stopped either the relative residual is smaller than $10^{-8}$ or  the maximal iteration number $5000$ is attained. 
\begin{figure}[H]
	\centering
	\includegraphics[width=450pt,height=350pt]{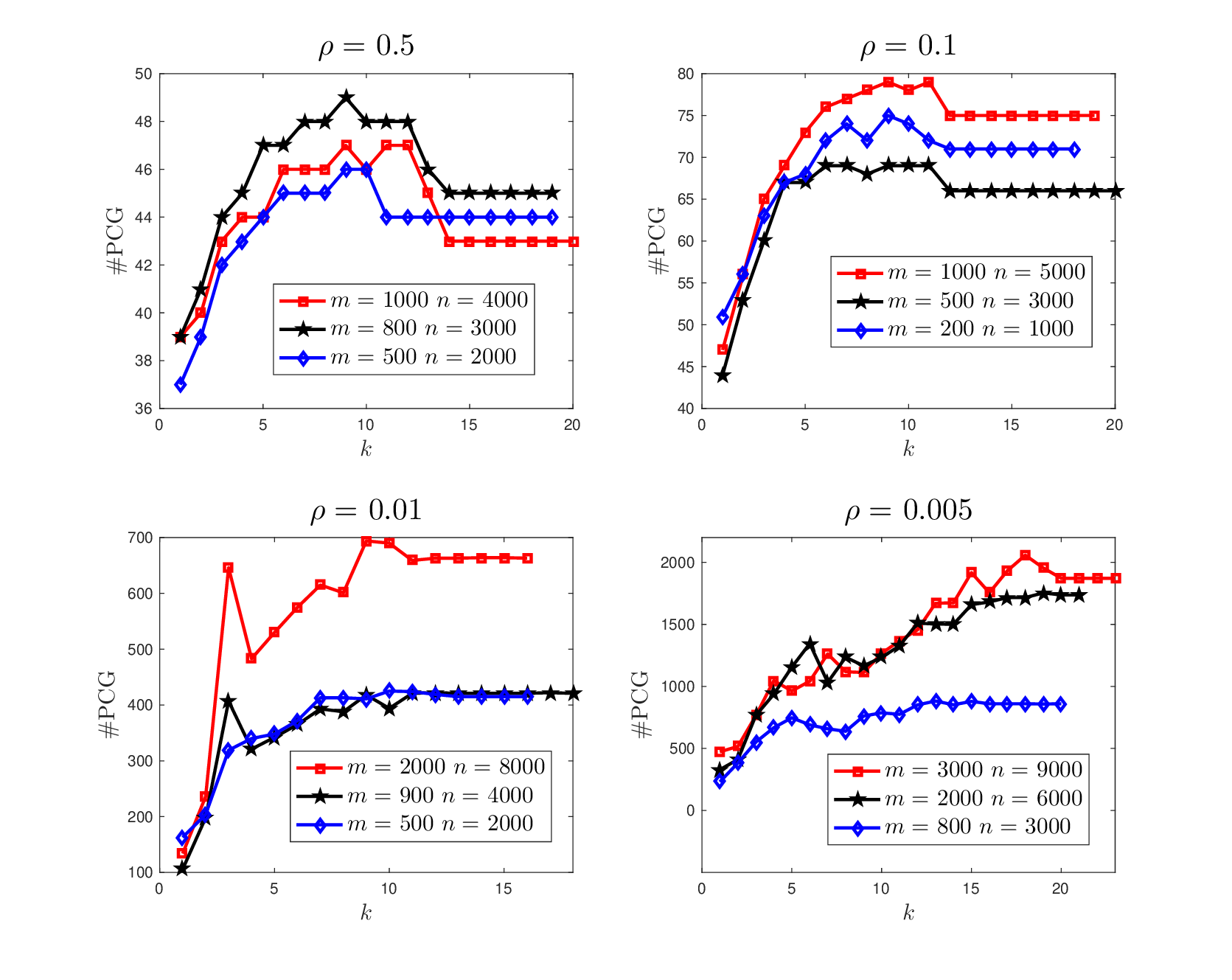}
	\caption{Averaged PCG iterations of Algorithm \ref{algo:Semi-PDPG-l1l2} for solving \eqref{eq:l1-l2} with different problem size and $\rho$.}
	\label{fig:Semi-PDPG-l1-l2-pcg}
\end{figure}

Computational results are reported in \cref{tab:l1-l2}, which includes (i) {\bf its}: the number 
of total iterations, (ii) {\bf SsN}: the number of the SsN iterations for the inner problem \cref{eq:iner-l1l2}, and (iii) {\bf time}: the running time (in seconds). 
To achieve the tolerance \cref{eq:res-kkt}, the number of iterations of Algorithm \ref{algo:Semi-PDPG-l1l2} is almost $k^* = 6\ln 10/\ln2\approx20$. 
This can be observed from \cref{tab:l1-l2}. However, as $\rho$ becomes small, the problem \cref{eq:l1-l2} itself is more degenerate and the 
number of iterations of the ALB method grows dramatically. 

\subsubsection{Performance of the PCG iteration}
\label{sec:PCG-l1l2}
From \cref{tab:l1-l2} we see that Algorithm \ref{algo:Semi-PDPG-l1l2} with PCG solver is slightly inferior than that with direct solver, both for total iteration number and running time. We now investigate the performance of the PCG iteration.

The linear system arising from step \ref{algo:d-l1l2} of Algorithm \ref{algo:Semi-PDPG-l1l2} is $JF_k(\lambda)d = -F_k(\lambda)$, where $JF_k(\cdot)= \beta_{k+1}I+\eta_k\mathcal A_0(\cdot)$ is defined by \cref{eq:JFk} and $\mathcal A_0(\cdot)$ is symmetric semi-positive definite. Note that $JF_k(\cdot)$ is always SPD but also nearly singular as $\beta_{k+1}\to 0$. Hence, the iteration number will increase as $k$ does. Fortunately, for large $k$, we may expect that $F_k(\cdot)$ is close to zero (as the algorithm converges) and the nearly singular property is not a serious problem.

Recall that we used the diagonal preconditioner, i.e., Jacobi iteration, and the terminal criterion is relative residual $\leqslant 10^{-8}$, with the maximal iteration number $5000$. In every $k$-th step of Algorithm \ref{algo:Semi-PDPG-l1l2}, we record the PCG iteration $\#_{k,j}$ of the $j$-th SsN iteration and obtain an averaged number $\#_k =\frac{1}{s_k} \sum_{j=1}^{s_k}\#_{k,j}$, where $s_k$ denotes the number of SsN iterations for solving the subproblem \cref{eq:iner-l1l2}.

In Figure \ref{fig:Semi-PDPG-l1-l2-pcg}, we plot the averaged PCG iterations of Algorithm \ref{algo:Semi-PDPG-l1l2} with the same problem size and $\rho$ used in \cref{tab:l1-l2}. As predicted above, due to the nearly singular property, the PCG iteration number grows up as $k$ increases but stays flat for large $k$. Moreover, it is not robust with respect to the problem size and $\rho$. 
\subsubsection{Restarting and warm-up}
Note that in the few starting steps, i.e., for small $k$, the SsN iteration may not achieve the desired tolerance $\nm{F_k(\lambda^j )}\leqslant \mathtt{SsN\_Tol}$ within $j_{\max} = 10$ iterations and the KKT residual ${\rm Res}(k)$ (cf.\cref{eq:res-kkt}) might not decay linearly while $\beta_k$ has already attained a small number, which makes the subproblem \cref{eq:iner-l1l2} degenerate. Hence, to ensure the stability, we adopt the restart technique. 
\begin{figure}[H]
	\centering
	\includegraphics[width=400pt,height=350pt]{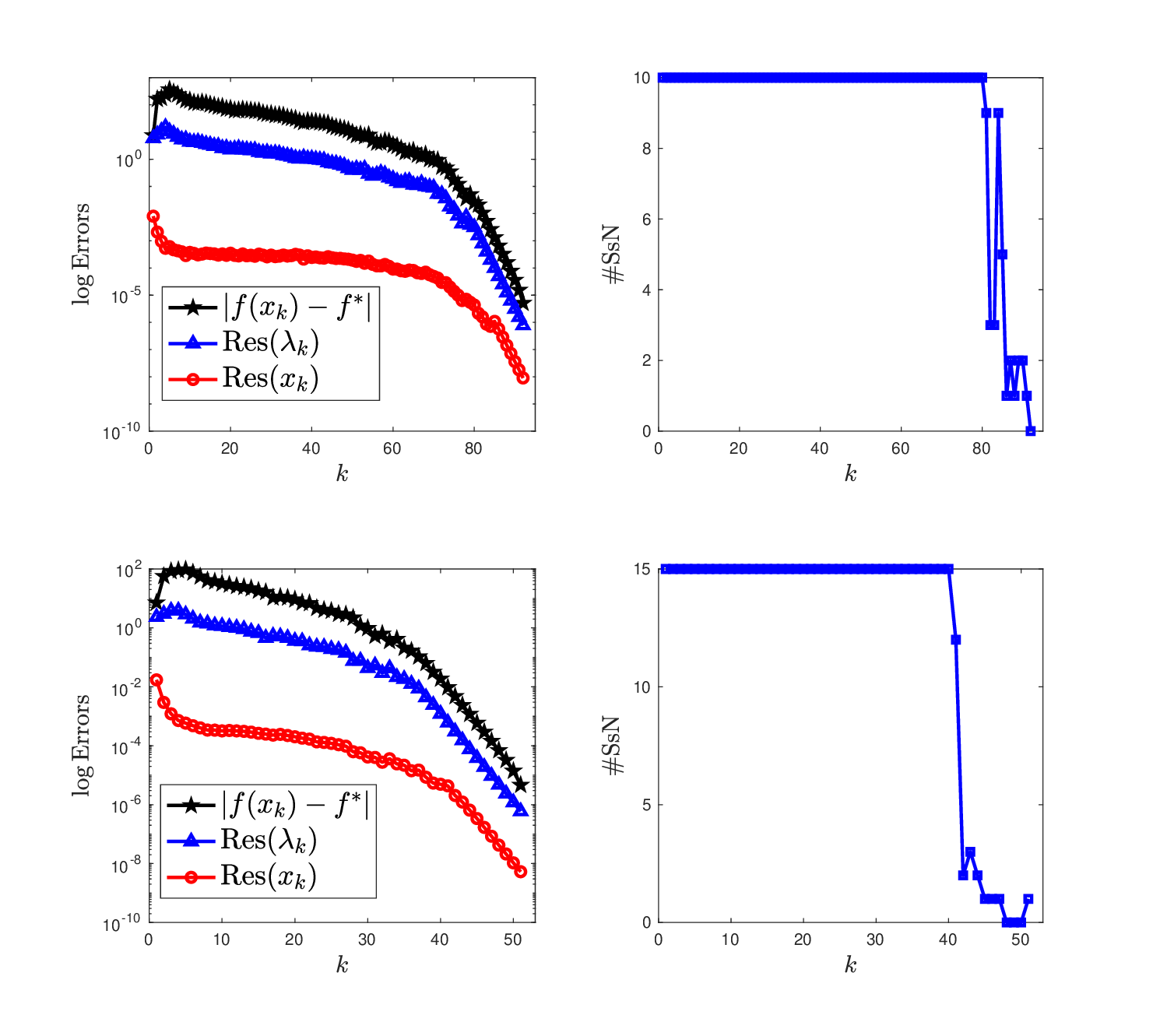}
	\caption{Performance of Algorithm \ref{algo:Semi-PDPG-l1l2} for solving \eqref{eq:l1-l2} with $m = 2000,\,n = 5000$ and $\rho = 0.0005$. The maximal iteration numbers of the SsN iteration for the top row and the bottom row are $j_{\max}=10$ and $j_{\max}=15$, respectively. The left part plots the decay behavior of the errors and the right part shows the number of SsN iteration in each step.}
	\label{fig:Semi-PDPG-l1-l2}
\end{figure}

We consider a more singular case $\rho = 0.0005$ and restart the algorithm whenever $\beta_k\leqslant 10^{-7}$ and the KKT residual ${\rm Res}(k)$ increases. From Figure \ref{fig:Semi-PDPG-l1-l2}, we observe that for this extreme case, (i) the total iteration number increases; (ii) in more than half of the total number of iterations, the errors decay slowly and the SsN iteration number attains its maximal value $j_{\max}$ (we set $j_{\max}=10$ for the top row and $j_{\max}=15$ for the bottom row), but after that, fast local linear convergence arises and the number of SsN iterations decreases.

As suggested by the results in Figure \ref{fig:Semi-PDPG-l1-l2}, a warm-up procedure might improve the performance of the algorithm and we show this in Figure \ref{fig:Semi-PDPG-l1-l2-warmup}, where the initial guess is obtained by running the ALB method 500 times. This works well indeed and the convergence behavior is much better than that in Figure \ref{fig:Semi-PDPG-l1-l2}.
\begin{figure}[H]
	\centering
	\includegraphics[width=400pt,height=160pt]{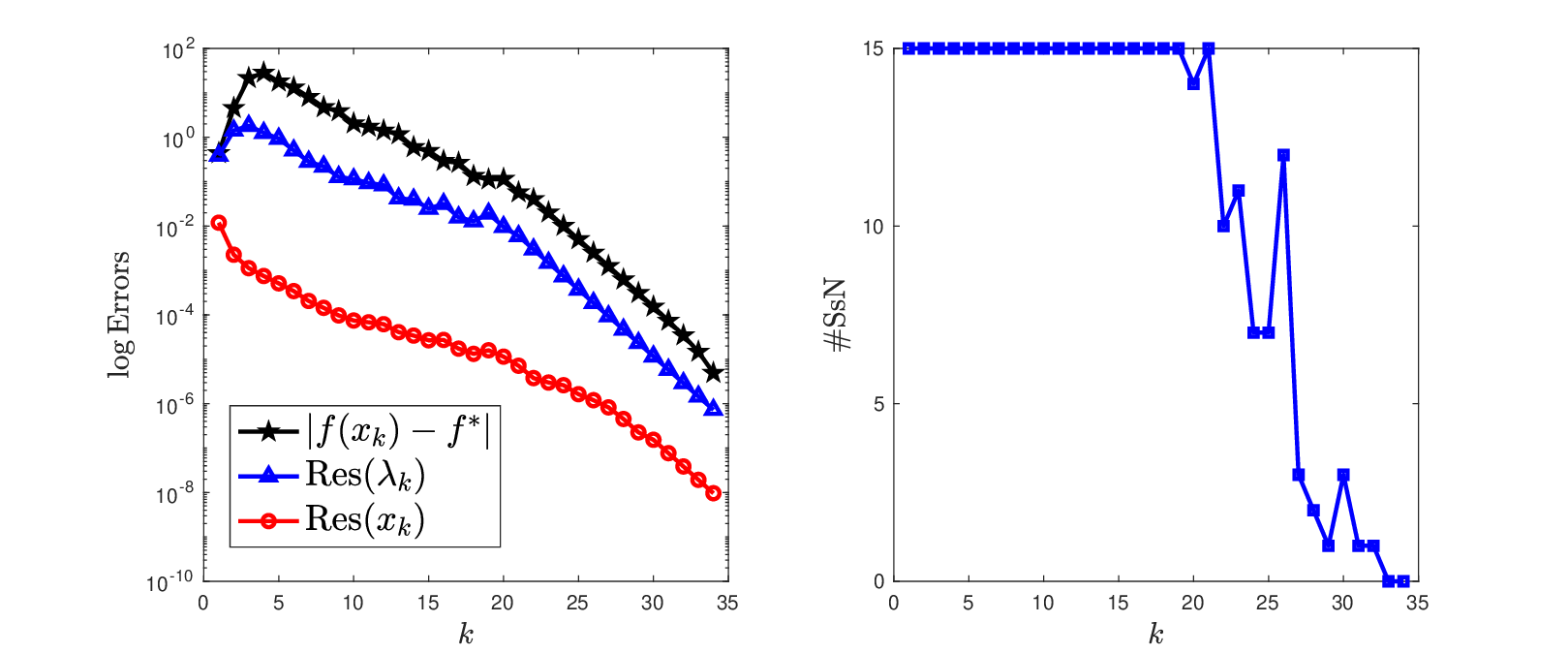}
	\caption{Performance of Algorithm \ref{algo:Semi-PDPG-l1l2} for solving \eqref{eq:l1-l2} with warm-up procedure. Here, we take $m = 2000,\,n = 5000$ and $\rho = 0.0005$, and the maximal number of the SsN iteration is $j_{\max}=15$. The initial guess is obtained via running the ALB method 500 times.}
	\label{fig:Semi-PDPG-l1-l2-warmup}
\end{figure}
\subsection{Total-variation based image denoising}
\label{sec:rof}
Given a noised image $g\in L^2(\Omega)$ with the domain $\Omega\subset\R^2$, the total variation based denoising model proposed by Rudin, Osher and Fatemi (ROF for short) \cite{rof_1992} reads as follows
\begin{equation}\label{eq:rof}
	\min_{u}\int_{\Omega}\snm{\nabla u}\dd x+ \frac{\rho}{2}\nm{u-g}^2_{L^2(\Omega)},
\end{equation}
where $\rho>0$ is the regularization parameter and $\snm{\nabla u}: = \sqrt{|\nabla_xu|^2+|\nabla_yu|^2}$.
\subsubsection{Discrete formulations}
In discrete setting, problem \cref{eq:rof} becomes
\begin{equation}\label{eq:rof-dis}
	\min_{U\in\R^{m\times n}}
	\sum_{i=1}^{m}\sum_{j=1}^{n}\sqrt{|\left(\mathcal D(U)\right)_{i,j,1} |^2+ |\left(\mathcal D(U)\right)_{i,j,2}|^2}
	+\frac{\rho}{2} \nm{U-\Xi}^2_F,
\end{equation}
where $\Xi\in\R^{m\times n}$ and $\mathcal D:\R^{m\times n}\to \R^{m\times n\times 2}$ denotes the discrete gradient operator, i.e., 
\[
\left(\mathcal D(U)\right)_{i,j,1} := 
\left\{
\begin{aligned}
	{}&U_{i+1,j}-U_{i,j}&&\text{if}\,i<m,\\
	{}&0&&\text{if}\,i=m,
\end{aligned}
\right.\quad\text{and}\quad
\left(\mathcal D(U)\right)_{i,j,2}:= 
\left\{
\begin{aligned}
	{}&U_{i,j+1}-U_{i,j}&&\text{if}\,j<n,\\
	{}&0&&\text{if}\,j=n,\\
\end{aligned}
\right.
\]
for all $1\leqslant i\leqslant m$ and $1\leqslant j\leqslant n$. Let ${\rm vec}(\times)$ be the vector expanded by the matrix $\times $ by its column. 
Then rewrite \cref{eq:rof-dis} as a composite problem
\begin{equation}\label{eq:min-u}
	\min_{u\in\R^{mn}}\,\psi(Au)+\frac{\rho}{2}\nm{u-\xi}^2 ,
\end{equation}
where $\xi = {\rm vec}(\Xi)$ and $A : = \begin{pmatrix}
	I_n\otimes D_m\\
	D_n\otimes I_m\\
\end{pmatrix}$, with the difference matrices $D_m$ and $D_n$ being defined such that $	(I_n\otimes D_m){\rm vec}(U) ={\rm vec}(\mathcal D(U)_{i,j,1}) $ and $	(	D_n\otimes I_m){\rm vec}(U) ={\rm vec}(\mathcal D(U)_{i,j,2}) $.

Let $\mathcal A = (-A,I)$ and introduce a function $\psi:\R^{2mn}\to\R$ by that
\[
\psi(\bm p): = \sum_{i=1}^{mn}\sqrt{ p_{i}^2+q_{i}^2}\quad\forall\,
\bm p=\begin{pmatrix}
	p\\q
\end{pmatrix}\in \R^{2mn}.
\]
Then \cref{eq:min-u} can be written as the standard form \cref{eq:min-f-Ax-b}:
\begin{equation}\label{eq:rof-dis-x-y-z}
	\min_{X=(u,\bm p)}f(X): = \frac{\rho}{2}\nm{u-\xi}^2+\psi(\bm p)  \quad{\rm s.t.}\,\mathcal AX= 0.
\end{equation}
\subsubsection{Accelerated primal-dual methods}
There are some well-known accelerated primal-dual methods for solving the discrete ROF model \cref{eq:rof-dis}. Here, we choose two baseline algorithms: the primal-dual hybrid gradient (PDHG) method \cite[Algorithm 2]{chambolle_first-order_2011} and the accelerated alternating direction method of multipliers (A-ADMM) \cite[Algorithm 2]{Xu2017}. Ergodic convergence rate $\mathcal O(1/k^2)$ is achieved by those two methods. For completeness, we list them as below.
\begin{itemize}
	\item {\bf PDHG method \cite[Algorithm 2]{chambolle_first-order_2011}}~~This method starts from the minimax formulation of \cref{eq:min-u}:
	\begin{equation}\label{eq:rof-dis-minmax-u}
		\min_{u\in\R^{mn}}\max_{\bm \lambda\in\R^{2mn}}\dual{Au,\bm \lambda}+\frac{\rho}{2} \nm{u-\xi}^2-\psi^*(\bm \lambda),
	\end{equation}
	where 
	\[
	\bm{\lambda} = \begin{pmatrix}
		v\\ w
	\end{pmatrix}\in\R^{2mn}\quad\text{and}\quad
	\psi^*(\bm \lambda): = \left\{
	\begin{aligned}
		{}&0&&\text{if }\sqrt{v_{i}^2+w_{i}^2}
		\leqslant 1\text{ for all }1\leqslant i\leqslant mn,\\
		{}&+\infty&&\text{ else}.
	\end{aligned}
	\right.
	\]
	More precisely, it reads as follows: given $\sigma_0=0,\,\bm \lambda_0\in\R^{2mn}$ and $u_{-1} = u_0\in\R^{mn}$, do the iteration
	\begin{equation}\label{eq:PDHG}
		\left\{
		\begin{aligned}
			{}&\bar u_{k} = u_{k} + \sigma_{k}(u_{k}-u_{k-1}),\\		
			{}&\bm \lambda_{k+1} = \prox_{\theta_k\psi^*}(\bm \lambda_k+\theta_kA\bar u_k),\\
			{}&u_{k+1} =\frac{u_k - \tau_k A^\top\bm \lambda_{k+1}}{1+\rho\tau_k} + \frac{\rho\tau_k\xi}{1+\rho\tau_k} ,\\		
			{}&\sigma_{k+1} = {}1/\sqrt{1+2\rho\tau_k},\,\tau_{k+1} = \sigma_{k+1}\tau_k,\,\theta_{k+1} = \theta_{k}/\sigma_{k+1},
		\end{aligned}
		\right.
	\end{equation}
	where $\tau_0\theta_0\nm{A}^2\leqslant 1$ with $\nm{A}^2\leqslant 8$ (cf. \cite{chambolle_algorithm_2004}). Thanks to Moreau's identity \cref{eq:moreau-id}, for all $\theta>0$, we have $\prox_{\theta \psi^*}(\bm \lambda) =(v\odot\sigma(\bm \lambda),w\odot\sigma(\bm \lambda))$, where $\sigma(\bm \lambda):=1-\tau(v,w)$ with $\tau(v,w)$ being defined by \cref{eq:tau}.
	\vskip0.2cm
	\item {\bf A-ADMM \cite[Algorithm 2]{Xu2017}}~~Applying this method to problem \cref{eq:rof-dis-x-y-z} leads to the iteration:  given $\theta\geqslant \nm{A}^2,\lambda_0=0,\,\bm p_0\in\R^{2mn}$ and $u_0\in\R^{mn}$, compute
	\begin{equation}\label{eq:A-ADMM}
		\left\{
		\begin{aligned}
			{}&		\theta_k = \frac{2\theta}{\rho(k+1)},\\
			{}&	\bm p_{k+1} = \prox_{\theta_k\psi}(Au_k-\theta_k\lambda_k),\\
			{}&u_{k+1} = \left(\rho\theta_kI+A^\top A\right)^{-1}\left(A^\top(\bm p_{k+1}+\theta_k\lambda_k)+\rho\theta_k\xi\right),\\		
			{}&	\lambda_{k+1} = \lambda_k + \frac{1}{\theta_k}\left(\bm p_{k+1}-Au_{k+1}\right),
		\end{aligned}
		\right.
	\end{equation}
	where $\prox_{\theta_k\psi}(\cdot)$ is defined by \cref{eq:prox-psi} and the inverse operation $\left(\rho\theta_kI+A^\top A\right)^{-1}$ can be realized via fast Fourier transform.
\end{itemize}
\subsubsection{Inexact implicit primal-dual method}
We apply our Algorithm \ref{algo:Im-PD} to problem \cref{eq:rof-dis-x-y-z} and obtain an inexact Im-PD method; see Algorithm \ref{algo:Im-PD-ROF}. For clarity, we provide some details about the proximal calculations. Given $a,b\in\R$ and $\theta>0$, define $\tau_\theta(a,b)\in\R$ and $\mathcal T_\theta(a,b)\in\R^{2\times 2}$ respectively by that
\[
\begin{aligned}
	\tau_\theta(a,b): ={}&1-\frac{\theta}{\max\left\{\theta,\sqrt{a^2+b^2}\right\}},\\
	\notag
	\mathcal T_\theta(a,b): = {}&
	\left\{
	\begin{aligned}
		{}&
		\tau_\theta(a,b)I + 
		\frac{1-\tau_\theta(a,b)}{a^2+b^2}\begin{pmatrix}
			a^2&ab\\ ab&b^2
		\end{pmatrix}
		&&\text{ if }\sqrt{a^2+b^2}\geqslant \theta,\\
		{}&O_{2\times 2}&&\text{else}.
	\end{aligned}
	\right.
\end{aligned}
\]
If $a,\,b\in\R^n$, then $\tau_\theta(a,b)\in\R^n$ can be understood as point wise operation:
\begin{equation}	\label{eq:tau}
	\tau_\theta(a,b): ={}{\bf 1}_n-\theta{\bf 1}_n\oslash\max\{\theta{\bf 1}_n,\sqrt{a\odot a + b\odot b}\}.
\end{equation}
For $\theta = 1$, we simply write $		\tau_\theta(a,b) = 		\tau(a,b)$.

For $X = (u,\bm p)\in\R^{3mn}$ and $\theta>0$, the proximal mapping of $f$ is given by $\prox_{\theta f}(X) = (	\frac{u+\rho\theta\xi}{1+\rho\theta},	\prox_{\theta\psi}(\bm p))$, where 
\begin{equation}\label{eq:prox-psi}
	\prox_{\theta\psi}(\bm p) = (p\odot\tau_\theta(p,q),q\odot\tau_\theta(p,q)).
\end{equation}
According to \cite[Chapter 7]{Facchinei2006}, $f$ is strongly semismooth and so is the nonlinear mapping $F_k(\cdot)$ defined by \cref{eq:Fk-rof}.
Moreover, a direct computation shows that $P_{\theta}(X)\in\partial \prox_{\theta f}(X) $ where 
\begin{equation}\label{eq:P_theta-X}
	P_{\theta}(X): = \begin{pmatrix}
		\frac{1}{1+\rho\theta}I&O\\
		O&T\\
	\end{pmatrix}\quad\text{with}\quad T = \begin{pmatrix}
		{\rm diag}(\tau_{11})&{\rm diag}(\tau_{12})\\
		{\rm diag}(\tau_{21})&{\rm diag}(\tau_{22})\\
	\end{pmatrix}.
\end{equation}
In \cref{eq:P_theta-X}, $T$ is block diagonal and $	\mathcal T_\theta(p_i,q_i) = \begin{pmatrix}
	(	\tau_{11})_i&(\tau_{12})_i\\
	(	\tau_{21})_i&(\tau_{22})_i\\
\end{pmatrix}$ for all $1\leqslant i\leqslant mn$. For $Y =(	s,\bm \lambda)\in\R^{3mn}$, it is not hard to find that $f^*(Y) = \frac{1}{2\rho}\nm{s}^2 + \dual{s,\xi}+\psi^*(\bm{\lambda})$, and thus the function \cref{eq:cal-Fk} becomes
\[
\begin{aligned}
	\mathcal F_k(\bm{\lambda})
	= {}&
	\frac{\beta_{k+1}}{2}\nm{\bm \lambda}^2-\dual{Z_k,\bm\lambda}
	+f^*(Y_k(\bm{\lambda})) +\frac{1}{2\theta_k}\nm{\prox_{\theta_k f}[Y_k(\bm{\lambda})]}^2,
\end{aligned}
\]
where $Z_k  ={} \beta_{k+1}\left(\bm \lambda_k-\beta_k^{-1}\mathcal AX_k\right) $ and $Y_k(\bm{\lambda})= X_k-\theta_k\mathcal A^{\top}\bm\lambda$.

As motivated by the first example (cf. Figure \ref{fig:Semi-PDPG-l1-l2-warmup}), in line \ref{algo:wram} of Algorithm \ref{algo:Im-PD-ROF}, we consider a warm-up step to provide a reasonable initial guess $(X_0,\bm{\lambda}_0)$ and therefore enhance the performance.
Besides, in step \ref{eq:d-rof}, the linear SPD system has special sparse structure that $T_k$ is a $2\times 2$ block matrix with each block being diagonal (see \cref{eq:P_theta-X}) and 
\[
AA^\top = \begin{pmatrix}
	H_{11}&H_{12}\\ H_{12}^\top&H_{22}
\end{pmatrix}=\begin{pmatrix}
	I_n\otimes D_mD_m^\top& D_n^\top\otimes D_m\\
	D_m^\top\otimes D_n&		D_nD_n^\top\otimes I_m\\
\end{pmatrix},
\]
where $H_{11}$ is block diagonal and both $H_{12}$ and $H_{22}$ are block tridiagonal. Hence, we consider the incomplete Cholesky factorization (cf. \cite[Chapter 10]{saad_iterative_2003}) as a preconditioner and apply preconditioned CG to step \ref{eq:d-rof} to obtain an approximation with relative residual $\leqslant 10^{-8}$.
\begin{algorithm}[H]
	\caption{Inexact Im-PD method for the discrete ROF model \cref{eq:rof-dis-x-y-z}}
	\label{algo:Im-PD-ROF}
	\begin{algorithmic}[1] 
		\REQUIRE  $\beta_0>0,\,\nu = 0.2$ and $\delta = 0.9$.
		\STATE Problem setting: $\rho>0$ and $\xi\in\R^{mn}$.
		\STATE Tolerances: $\mathtt{KKT\_Tol} = 10^{-6}$ and 
		$\mathtt{SsN\_Tol} =10^{-8}$.
		\STATE Perform a warm-up step to obtain: $X_0 = (u_0,\bm p_0)
		\in\R^{3mn}$ and $\bm\lambda_0\in\R^{2mn}$.	\label{algo:wram}
		\FOR{$k=0,1,\ldots$}
		\STATE Choose the step size $\alpha_k>0$ and update $\displaystyle \beta_{k+1}= \beta_k/(1+\alpha_k)$.
		\STATE Set $\displaystyle 		\theta_{k}={}\alpha_k/\beta_{k}$ and $\rho_k =1/( 1+\rho\theta_k)$.
		\STATE Set $\displaystyle Z_k  ={} \beta_{k+1}\left(\bm \lambda_k-\beta_k^{-1}\mathcal AX_k\right)$.
		\STATE Solve $\bm \lambda_{k+1}$ from the nonlinear equation
		\begin{equation}\label{eq:Fk-rof}
			F_{k}(\bm \lambda):=	\beta_{k+1}\bm \lambda -\mathcal A\proxi_{\theta_k f}\left(X_k-\theta_k\mathcal A^{\top}\bm \lambda\right)- Z_k=0
		\end{equation}
		via the following SsN iteration with the initial guess $\bm\lambda = \bm\lambda_k$:
		\WHILE[SsN iteration]{$\nm{F_k(\bm{\lambda})}>\mathtt{SsN\_Tol}$}
		\STATE Compute $Y_k = X_k - \theta_k\mathcal A^\top\bm\lambda$.
		\STATE Find $		P_{k}(Y_k) = \begin{pmatrix}
			\rho_kI&O\\O&T_k
		\end{pmatrix}\in\partial \prox_{\theta_k f}(Y_k)$ via \cref{eq:P_theta-X}.
		\STATE Compute $JF_k(\bm{\lambda}) =\beta_{k+1}I+\theta_kT_k +\rho_k \theta_kAA^\top$.
		\STATE Solve $JF_k(\bm{\lambda})\bm d = -F_k(\bm{\lambda})$ approximately via preconditioned CG.\label{eq:d-rof}
		\STATE Find the smallest integer $r\in\mathbb N_+$ such that $		\mathcal F_k(\bm\lambda+\delta^r\bm d)\leqslant \mathcal F_k(\bm \lambda)+\nu\delta^r\dual{F_k(\bm \lambda),\bm d}$.
		\STATE Update $\bm{\lambda} = \bm\lambda+\delta^r\bm d$.
		\ENDWHILE		
		\STATE Update $\bm{\lambda}_{k+1} = \bm{\lambda}$ and $X_{k+1}	 = \prox_{\theta_k f}\left(X_k-\theta_k \mathcal A^{\top}\bm\lambda_{k+1}\right)$.
		\IF{${\rm Res}(k)\leqslant \mathtt{KKT\_Tol}$}
		\STATE {\bf break}
		\ENDIF		
		\ENDFOR
	\end{algorithmic}
\end{algorithm}
\subsubsection{Numerical results}
We adopt four benchmark images from the literature: $\mathtt{barb}$, $\mathtt{boat}$, $\mathtt{cameraman}$ and $\mathtt{lena}$. These images are noised with standard normal distribution. Note that both \cref{eq:rof-dis-minmax-u,eq:rof-dis-x-y-z} admit the same optimality condition 
\[
\left\{
\begin{aligned}
	0=	{}&\rho(u^*-\xi) - A^\top\bm \lambda^*\\
	0\in		{}&\bm \lambda^*+\partial \psi(\bm p^*)\\
	0=		{}&	\bm p^* -Au^* \\
\end{aligned}
\right.\quad\Longleftrightarrow\quad
\left\{
\begin{aligned}
	0=	{}&\rho(u^*-\xi) - A^\top\bm \lambda^*\\
	0=	{}&
	\bm p^* - \prox_{ \psi}(\bm p^*-\bm \lambda^*)\\
	0=			{}&	\bm p^* -Au^* =0
\end{aligned}
\right..
\]
Hence, we consider the stopping criterion:
\begin{equation}\label{eq:res-rof}
	{\rm Res}(k) := \max\left\{{\rm Res}(u_k),\,{\rm Res}(\bm p_k),\,{\rm Res}(\bm \lambda_k)\right\}\leqslant \mathtt{KKT\_Tol} = 10^{-6},
\end{equation}
where the relative KKT residuals are defined by 
\[\small
{\rm Res}(u_k):= \frac{\nm{\rho(u_k-\xi) - A^\top\bm \lambda_k}}{1+\nm{\xi}},\,
{\rm Res}(\bm p_k):=	\frac{\nm{\bm p_k - \prox_{ \psi}(\bm p_k-\bm \lambda_k)}}{1+\nm{\bm p_k}}
\,\text{and}\,
{\rm Res}(\bm \lambda_k):=	\frac{\nm{\bm p_k - Au_k}}{1+\nm{\bm p_k}}.
\]
For all methods, the maximal iteration number is $k_{\max} = 1e5$. For inexact Im-PD (i.e. Algorithm \ref{algo:Im-PD-ROF}), we run the A-ADMM with 50 steps to obtain an initial guess $(u_0,\bm p_0,\bm{\lambda}_0)$ with $ \max\left\{{\rm Res}(u_0),\,{\rm Res}(\bm p_0),\,{\rm Res}(\bm \lambda_0)\right\}\approx 10^{-2}$ and choose the step size $\alpha_k = 1+\sigma$ where $\sigma$ obeys the uniform distribution on $[0,1]$. Then by \cref{thm:conv-im}, we have the linear rate $\varrho^k$ with $\varrho = \mathbb E[\frac{1}{1+\alpha_k}]=\int_{0}^{1}\frac{1}{2+\sigma}\dd\sigma=\ln3/2$ and the required iteration number for \cref{eq:res-rof} is about $k^*=-4\ln 10/\ln\varrho\approx 10$.
\renewcommand\arraystretch{1.2}
\begin{table}[H]
	\centering
	\caption{Performances of Algorithm \ref{algo:Im-PD-ROF}, PDHG \cref{eq:PDHG} and 
		A-ADMM \cref{eq:A-ADMM} for solving \cref{eq:rof-dis}. }
	\label{tab:rof}
	\small\setlength{\tabcolsep}{1.5pt}
	\begin{tabular}{ccccccccccccccccc}
		\toprule
		&&& &&  &\multicolumn{4}{c}{Inexact Im-PD (Algorithm \ref{algo:Im-PD-ROF})}&\phantom{a} &  \multicolumn{2}{c}{A-ADMM \cref{eq:A-ADMM}}&\phantom{a} &  \multicolumn{3}{c}{PDHG \cref{eq:PDHG}}\\
		\cmidrule{7-10}			\cmidrule{12-13} \cmidrule{15-17} 
		&&$m=n$ & &$\rho$ &&its& SsN&warm-up(sec)&time(sec)&    & its   & time(sec)&    & its  &Res$(k_{\max})$  & time(sec)\\						
		\midrule
		\multirow{2}{*}{ $\mathtt{barb}$}			
		&&  \multirow{2}{*}{512}  &&  50  &&  10  & 182  & 47.11 &  830.53 &&  1572  &  1497.40  &&  \multirow{2}{*}{$10^5$} &  5.15e-06 &  5169.55 \\ 
		&&    &&  150  &&  10  & 141  & 41.76 &  568.34 &&  3445  &  3192.33  &&    &  3.59e-06 &  5213.23 \\ 
		\midrule
		\multirow{2}{*}{ $\mathtt{boat}$}		
		&&  \multirow{2}{*}{512}  &&  40  &&  9  & 84  & 54.53 &  457.05 &&  1300  &  1145.84  &&  \multirow{2}{*}{$10^5$} &  5.49e-06 &  5228.70 \\ 
		&&    &&  180  &&  10  & 141  & 44.34 &  611.90 &&  3866  &  3316.70  &&     &  3.42e-06 &  5223.88 \\ 
		\midrule
		\multirow{2}{*}{ $\mathtt{cameraman}$}		
		&&  	\multirow{2}{*}{256}   &&  20  &&  7  & 52  & 8.48 &  78.91 &&  724  &  124.58  &&  \multirow{2}{*}{$10^5$}&  7.77e-06 &  1299.35 \\ 
		&&    &&  100  &&  10  & 81  & 8.36 &  70.26 &&  2575  &  448.94  &&  &  4.10e-06 &  1262.76 \\ 
		\midrule
		\multirow{2}{*}{ $\mathtt{lena}$}	
		&&  \multirow{2}{*}{256}   &&  50  &&  9  & 111  & 8.55 &  117.86 &&  1554  &  288.23  &&   \multirow{2}{*}{$10^5$}&  5.29e-06 &  1229.09 \\ 
		&&    &&  200  &&  11  & 157  & 8.47 &  158.41 &&  4099  &  758.60  &&     &  3.55e-06 &  1210.58 \\ 
		\bottomrule
	\end{tabular}
\end{table}
Computational results are summarized in \cref{tab:rof}, including the number of iterations ({\bf its}) and running time ({\bf time}). For Inexact Im-PD, we also report the total number of SsN iterations ({\bf SsN}) and the time used for initialization ({\bf warm-up}). For all cases, PDHG has not achieved the tolerance \cref{eq:res-rof} within the maximal iteration number $k_{\max} = 10^5$, and we also record the KKT residual Res$(k_{\max})$ at the last iterate. As we can see, Algorithm \ref{algo:Im-PD-ROF} outperforms much better than other two methods and the total iteration number is almost $10$, as expected above. Particularly, we observe that A-ADMM is more efficient than PDHG. 

Moreover, in Figure \ref{fig:Im-PD-rof-pcg}, we plot the averaged PCG iteration number of Algorithm \ref{algo:Im-PD-ROF} for all cases. Similar as before (cf. Figure \ref{fig:Semi-PDPG-l1-l2-pcg}), it increases along with the iteration. Therefore, this deserves further study for more robust and efficient linear solvers such as algebraic multilevel methods \cite{lee_robust_2007,xu_algebraic_2017}. 
\section{Concluding Remarks}
\label{sec:conclu}
In this work, we introduce a novel dynamical system, called primal-dual flow, for 
solving affine constrained convex optimization. The current model is a modification of the standard saddle-point dynamics. In continuous level, exponential decay of a tailored Lyapunov function is established. Then, in discrete level, primal-dual type algorithms are obtained from proper time discretizations of the 
presented primal-dual flow and nonergodic convergence rates 
are established via a unified discrete Lyapunov function.

The proposed methods adopt dynamically changing parameters and the subproblem with respect to the multiplier is solved by the semi-smooth Newton iteration. This can be quite efficient provided that the problem has nice properties such as semi-smoothness and sparsity, as showed by numerical results of the $l_1$-$l_2$ problem and the total-variation based denoising model.

To the end, we list several ongoing works. First, well-posedness (existence and uniqueness) of the primal-dual flow  system \cref{eq:pdf-x-non} is an interesting topic. Also, the exponential decay property \cref{eq:pdf-sys-EG} and weak convergence of the trajectory under general nonsmooth setting deserves future investigations. Besides, rigorous convergence rate analysis with inexact computation and restart technique requires further attentions.
\begin{figure}[H]
	\centering
	\includegraphics[width=450pt,height=380pt]{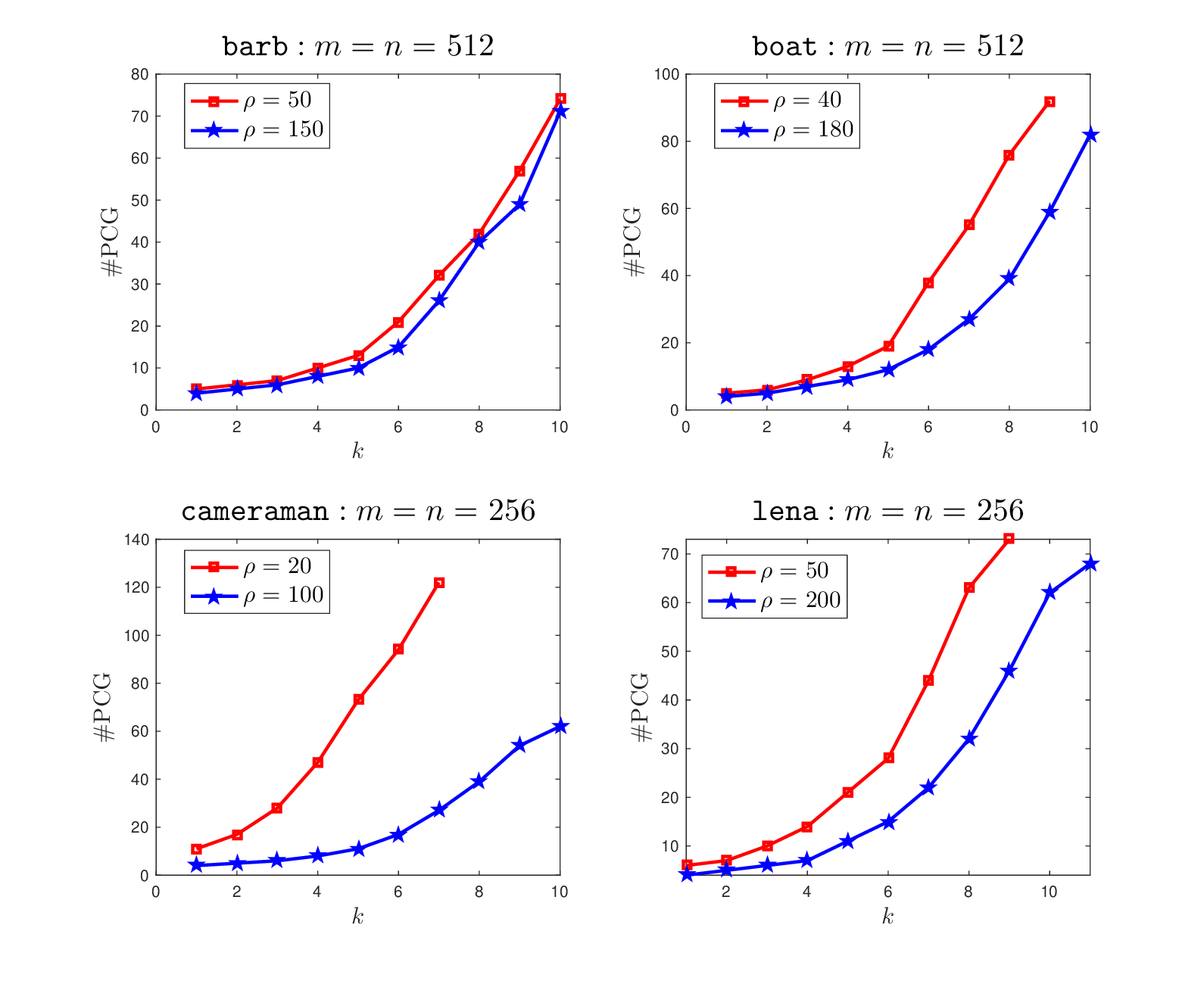}
	\caption{Averaged PCG iterations of Algorithm \ref{algo:Im-PD-ROF} for solving \cref{eq:rof-dis} with different noised input images and regularization parameters.}
	\label{fig:Im-PD-rof-pcg}
\end{figure}
\section*{Acknowledgments}
This work was supported by the NSFC project 11625101. 
The author would like to thank Professor Jun Hu for useful comments and advices. Besides, the author want to thank the two anonymous reviewers, as the manuscript was greatly benefit from their invaluable suggestions. 

\appendix
\section{An Over-Relaxation Perspective}
\label{app:pdhg}
In Section \ref{sec:pd-flow}, we introduced our primal-dual flow by adding the extra term $x'$, which is motivated from the disappointing estimate in \cref{lem:pdf-try-dE-G} and leads to the desired exponential decay, and later in Section \ref{sec:simple-model}, we provided an equilibrium illustration to show further the positive effects of this correction. 

To better understand the modification from the saddle-point system \cref{eq:sp} to our new model \cref{eq:pdf-x}, in this appendix, by using the PPA-like interpretation \cite{he_convergence_2012}, we give a discrete over-relaxation perspective, which indicates somewhat subtle connection with the hidden symmetrization from the Arrow--Hurwicz algorithm \cite{zhu_ecient_2008} to the PDHG method \cite{chambolle_first-order_2011}. We hope this provides a more reasonably intrinsic explanation. 

The Arrow--Hurwicz algorithm can be applied to \cref{eq:min-f-Ax-b} and reads as 
\begin{equation}\label{eq:pdhg-0}
	\left\{
	\begin{aligned}
		x_{k+1}= {}&\mathop{\argmin}_{x\in\R^n}\left\{\mathcal L(x,\lambda_k)+\frac{1}{2r}\nm{x-x_k}^2\right\},\\
		\lambda_{k+1}= {}&\mathop{\argmax}_{\lambda\in\R^m}\left\{\mathcal L(x_{k+1},\lambda)-\frac{1}{2\tau}\nm{\lambda-\lambda_k}^2\right\},
	\end{aligned}
	\right.
\end{equation}
with step sizes $r,\,\tau>0$. It also corresponds to a semi-implicit discretization for \cref{eq:sp}:
\begin{equation}\label{eq:semi}
	\left\{
	\begin{aligned}
		\frac{x_{k+1}-x_k}{r}\in {}&-\partial_x\mathcal L(x_{k+1},\lambda_k),\\
		\frac{		\lambda_{k+1}-		\lambda_{k}}{\tau}= {}&\nabla_\lambda\mathcal L(x_{k+1},\lambda_{k+1}).
	\end{aligned}
	\right.
\end{equation}

Following \cite{he_convergence_2012} and \cite[Chapter 8]{clason_nonsmooth_2020}, we use the PPA-like interpretation to demonstrate the lack of symmetry of the Arrow--Hurwicz algorithm. 
Introduce
\[
\begin{aligned}
	\small
	Z = \begin{pmatrix}
		x\\\lambda
	\end{pmatrix},\quad 
	M(Z) = \begin{pmatrix}
		\partial f(x) + A^\top\lambda\\
		b-	Ax
	\end{pmatrix}
	\quad\text{and}\quad
	Q = \begin{pmatrix}
		I/r & -A^\top\\ O& I/\tau
	\end{pmatrix},
\end{aligned}
\]
where the maximally monotone operator $M$ has been defined in \cref{eq:M}.
We then have the variational inequality characterization for \cref{eq:pdhg-0} (or \cref{eq:semi}):
\[
\begin{aligned}
	\dual{Q(Z_{k+1}-Z_k)+M(Z_{k+1}),Z-Z_{k+1}}
	\geqslant 0\quad\forall\,Z\in\R^{n+m}.
\end{aligned}
\]
Taking $Z = Z^*\in\Omega^*$ and utilizing the fact: $	0\in M(Z^*)$, we find that
\begin{equation}\label{eq:est-0}
	\begin{aligned} 
		{}&
		\frac{1}{2}\nm{Z_{k+1}-Z^*}_Q^2-\frac{1}{2}\nm{Z_{k}-Z^*}_Q^2 \\
		= {}&\dual{Q(Z_{k+1}-Z_k),Z_{k+1}-Z^*}  -\frac{1}{2}\nm{Z_{k+1}-Z_k}_Q^2
		+\frac{1}{2}\big\langle(Q^\top-Q)(Z_{k+1}-Z^*),Z_k-Z^*\big\rangle\\
		\leqslant {}&\underbrace{-\dual{M(Z_{k+1}),Z_{k+1}-Z^*}}_{\leqslant 0} -\frac{1}{2}\nm{Z_{k+1}-Z_k}_Q^2
		+\frac{1}{2}\big\langle(Q^\top-Q)(Z_{k+1}-Z^*),Z_k-Z^*\big\rangle\\	
		\leqslant {}	{}&-\frac{1}{2}\nm{Z_{k+1}-Z_k}_Q^2 +\frac{1}{2}\big\langle\underbrace{(Q^\top-Q)}_{\neq 0}(Z_{k+1}-Z^*),Z_k-Z^*\big\rangle.
	\end{aligned}
\end{equation}
As $Q$ is not symmetric, the last term makes it hard to obtain the descent estimate,
and what's even worse, the scheme \eqref{eq:pdhg-0} is not necessarily convergent \cite{he_algorithmic_2017}.

The PDHG method of Chambolle and Pock introduces a parameter $\theta\in[0,1]$ and becomes
\begin{equation}\label{eq:pdhg}
	\left\{
	\begin{aligned}
		x_{k+1}= {}&\mathop{\argmin}_{x\in\R^n}\left\{\mathcal L(x,\lambda_k)+\frac{1}{2r}\nm{x-x_k}^2\right\},\\
		\lambda_{k+1}= {}&\mathop{\argmax}_{\lambda\in\R^m}\left\{\mathcal L(x_{k+1}+\theta(x_{k+1}-x_k),\lambda)-\frac{1}{2\tau}\nm{\lambda-\lambda_k}^2\right\},
	\end{aligned}
	\right.
\end{equation}
which is also equivalent to
\begin{equation}\label{eq:pdhg-dis}
	\left\{
	\begin{aligned}
		\frac{x_{k+1}-x_k}{r}\in{}&-\partial_x\mathcal L(x_{k+1},\lambda_k),\\
		\frac{		\lambda_{k+1}-		\lambda_{k}}{\tau}= {}&\nabla_\lambda\mathcal L(x_{k+1}+\theta(x_{k+1}-x_k),\lambda_{k+1}).
	\end{aligned}
	\right.
\end{equation}
Comparing  this with the previous discretization \eqref{eq:semi}, we observe the additional extrapolation term $x_{k+1}-x_k$. For the case $\theta=1$, we have ergodic convergence rate $\mathcal O(1/k)$ under the condition $r\tau \nm{A}^2<1$. Moreover, applying the above PPA-like framework to the PDHG method (with $\theta=1$), one observes that the estimate \cref{eq:est-0} is now improved to
\[
\frac{1}{2}\nm{Z_{k+1}-Z^*}_{\widehat{Q}}^2-\frac{1}{2}\nm{Z_{k}-Z^*}_{\widehat{Q}}^2
\leqslant {}	{}-\frac{1}{2}\nm{Z_{k+1}-Z_k}_{\widehat{Q}}^2\quad\text{with}\quad 	{\widehat{Q}} = \begin{pmatrix}
	I/r & -A^\top\\ -A& I/\tau
\end{pmatrix},
\]
where $\widehat{Q}$ is a symmetrization of $Q$, due to the over-relaxation $x_{k+1}-x_k$. 

Surprisingly, instead of the original saddle-point system \cref{eq:sp}, the PDHG method \cref{eq:pdhg} is more likely a time discretization (cf. \cref{eq:pdhg-dis}) for the modified model 
\[
\left\{
\begin{aligned}
	{}&		x'  =-\nabla_x \mathcal L(x,\lambda),\\
	{}&\lambda'  ={}\nabla_\lambda \mathcal L(x+x',\lambda),
\end{aligned}
\right.
\]
which differs from our primal-dual flow \cref{eq:pdf-x} only in the time scaling parameters. In conclusion, the extra derivative $x'$ in $\nabla_\lambda \mathcal L(x+x',\lambda)$ corresponds to discrete over-relaxation $x_{k+1}-x_k$ in PDHG, which possibly brings hidden symmetrization.

\end{document}